\documentclass[11pt]{article}
\usepackage{geometry}
\usepackage{color}
\usepackage{indentfirst}
\usepackage{algorithm}
\usepackage{algorithmic}
\usepackage{amsthm}
\usepackage{subfigure}
\usepackage{graphicx}
\usepackage{amsmath}
\usepackage{amssymb}
\usepackage{epstopdf}
\usepackage{arydshln}
\usepackage{blkarray}
\usepackage{calc}

\usepackage [latin1]{inputenc} 
\usepackage[T1]{fontenc} 
\usepackage{multirow}
\usepackage{fancyhdr}
\usepackage[title]{appendix}
\usepackage{enumerate}
\usepackage{tikz}
\usetikzlibrary{trees}
\usepackage{pifont}
\usepackage[numbers,sort&compress]{natbib}
\usepackage{caption}
\usepackage{hyperref}
\newtheorem{Lemma}{\textbf{Lemma}}
\newtheorem{Theorem}{\textbf{Theorem}}
\newtheorem{Example}{\textbf{Example}}

\newtheorem{Assumption}{\textbf{Assumption}}

\newtheorem{remark}{Remark}

\date{\today}

\title{Quantum-inspired algorithm for truncated total least squares solution
\thanks{This work was supported by the National Key Research and Development Program of China No.\ 2021YFA1000600, the National Natural Science Foundation of China under Grant No.\ 11571265.}}

\author{Qian Zuo\footnotemark[2]\
\and Yimin Wei\footnotemark[3]
\and Hua Xiang\footnotemark[2] \footnotemark[4]}

\begin{document}
\maketitle
\renewcommand{\thefootnote}{\fnsymbol{footnote}}
\footnotetext[2]{School of Mathematics and Statistics, Wuhan University, Wuhan 430072, China.}
\footnotetext[3]{School of Mathematical Sciences and Shanghai Key Laboratory of Contemporary Applied Mathematics, Fudan University, Shanghai 200433, China.}
\footnotetext[4]{Hubei Key Laboratory of Computational Science, Wuhan University, Wuhan 430072, China.}
\footnotetext{E-mail addresses: {\tt zuoqian@whu.edu.cn(Q.Zuo)},\, {\tt ymwei@fudan.edu.cn(Y.Wei)},\, {\tt hxiang@whu.edu.cn(H.Xiang)}.}

\begin{abstract}
Total least squares (TLS) methods have been widely used in data fitting. Compared with the least squares method, for TLS problem we takes into account not only the observation errors, but also the errors in the measurement matrix. This is more realistic in practical applications. 
For the large-scale discrete ill-posed problem $Ax \approx b$, we introduce the quantum-inspired techniques to approximate the truncated total least squares (TTLS) solution. We analyze the accuracy of the quantum-inspired truncated total least squares algorithm and perform numerical experiments to demonstrate the efficiency of our method.
{\small
\begin{description}
\item[{\bf Keywords:}] Total least squares problems, truncated total least squares, sample model, randomized algorithms, quantum-inspired algorithm
\end{description}
}
\end{abstract}

\section{Introduction}
We consider the discrete ill-posed linear system $Ax \approx b$, $A \in \mathbb{R}^{m \times n}$ and $m \geq n$. The total least squares (TLS) problem can be formulated as \cite{SVHJV91}
\begin{equation}\label{QI:TLS:Eq:tls}
 \{x_{\rm TLS}, E_{\rm TLS}, f_{\rm TLS}\} := \mathop{\rm argmax} \limits_{x,E,f} \|E,F\|_{F} \ \ \ \ s.t. \ \ (A+E)x = b + f,
\end{equation}
where $E$ denotes the errors in the observation matrix $A$, $f$ denotes the errors in the observation vector $b$, and $\| \cdot \|_{F}$ represents the Frobenius matrix form.

Golub and Van Loan first proposed the concept of total least squares problem in \cite{GVL80}, and it also has been known as signal processing, automatic control, physics, astronomy, biology, statistics, economics, etc. \cite{SVHJV91,SVHPL02,BH91,SVH97,SVH04}. Compared with the least squares (LS) problem, the TLS solution not only includes the errors $f$ of observation vector $b$, but also errors $E$ from the measurement matrix $A$ of the variables. The more relationships between the LS and TLS problems can be seen in \cite{SVHJV91,GVL80,WMS92}. Besides, for the condition numbers of the total least squares problem have been considered in \cite{BG11,LJ11,XXW17,ZLWQ09,CT09,GTI13,MDB21,JL13,ZY19}. The concept of the core problem is proposed by Paige and Strako\v{s} in \cite{PS06} and used to find the minimum norm solution of the TLS problem.
For the TLS problems with multiple right-hand sides, i.e., $AX \approx B$, it has been considered in \cite{SVHJV91,HPSSV11,HPS13,LLW22,LJW22,ZMW17,HPSSV15,MZW20,HPS16}.
To name just a few, Van Huffel and Vandewalle give the generalizations to the nongeneric and multiple right-hand sides problems in \cite{SVHJV91}.
In \cite{LJW22}, Liu, Jia and Wei derive an explicit solution for the multidimensional total least squares problem with linear equality constraints (TLSE) problem and present the general formulae of its condition numbers and their computable upper bounds. In this work, we focus on the TLS problem with single right-hand sides, i.e., $Ax \approx b$.

The classical method to compute the TLS problem is based on SVD of the augment matrix $C =[A,\,b]$. Let $A \in \mathbb{R}^{m \times n}$ and $b \in \mathbb{R}^{m}$ with $m \geq n$. Suppose that the SVDs of $C$ and $A$ as follows, respectively.
\begin{equation}\label{QI:TLS:Eq:SVDAb}
    \begin{aligned}
        U^{T} C V & = \Sigma = {\rm diag}\{\sigma_1, \cdots, \sigma_t\}, \\
        {U}^{T}_{A} A V_{A} & = \Sigma_{A} = {\rm diag}\{\sigma^{A}_1, \cdots, \sigma^{A}_n\},
    \end{aligned}
\end{equation}
where $t = {\rm min}\{m,n+1\}$, $U \in \mathbb{R}^{m \times (n+1)}$, $V \in \mathbb{R}^{(n+1) \times (n+1)}$ are orthonormal and $\Sigma \in \mathbb{R}^{(n+1) \times (n+1)}$ is a diagonal matrix. Partition $U$, $V$ and $\Sigma$ as follows.
\begin{equation}\label{QI:TLS:Eq:SVDAb}
U = [U_1, u_{n+1}], \ \
V=
	\left[
	\begin{array}{cc}
	V_{11}& v_{12} \\
	v_{21}& v_{22}
	\end{array}
	\right], \ \
\Sigma=
	\left[
	\begin{array}{cc}
	\Sigma_{1}& 0 \\
	0         & \sigma_{n+1}
	\end{array}
	\right],
\end{equation}
where $u_{n+1} \in \mathbb{R}^{m}$, $v_{12} \in \mathbb{R}^{n}$, $v_{21}^{T} \in \mathbb{R}^{n}$. The genericity condition is
$\sigma^{A}_n > \sigma_{n+1}$, which ensures the existence and the uniqueness of the TLS solution \cite{GVL80}. The TLS solution is
\begin{equation}\label{QI:TLS:Eq:so1}
     x_{\rm TLS} = - \frac{v_{12}}{v_{22}}, \ \
     \left[
       \begin{array}{c}
         x_{\rm TLS} \\
         -1 \\
       \end{array}
     \right]
     = - \frac{v_{n+1}}{v_{22}}.
\end{equation}

However, in practice applications, if $\sigma^{A}_n$ coincides with $\sigma_{n+1}$, the TLS problem may still have a solution, but the solution is no longer unique. In this paper, we focus on the minimum norm TLS solution in the case where
\begin{equation}\label{QI:TLS:Eq:sqsn}
     \sigma^{A}_q > \sigma_{q+1} = \sigma_{q+2} = \cdots = \sigma_{n+1}, \ \ q \leq n,
\end{equation}
which implies that $\sigma_q > \sigma_{q+1} = \sigma_{q+2} = \cdots = \sigma_{n+1}$. Based on Eq.\eqref{QI:TLS:Eq:SVDAb}, we use a modified partition form in the following
\begin{equation}\label{QI:TLS:Eq:PVq}
\begin{array}{c@{\hspace{-5pt}}l}
V =
	\left[
	\begin{array}{cc}
	V_{11}& \quad V_{12} \\
	v_{21}& \quad v_{22}
	\end{array}
	\right]
& \begin{array}{l}
\, n\\
\, 1 ,
\end{array}\\
\begin{array}{cc}
\qquad \quad q &
\, n+1-q
\end{array}
\end{array}
\end{equation}
where $V_{11} \in \mathbb{R}^{n \times q}$, $v_{21}^{T} \in \mathbb{R}^{q}$, and $v_{22}$ is a row vector with full rank. According to the theorem 3.7 in \cite{SVHJV91}, the classical TLS solution has more than one solution in the condition \eqref{QI:TLS:Eq:sqsn}. Based on the the partition form of $V$ in \eqref{QI:TLS:Eq:PVq}, we can select the minimum norm TLS solution. The determination of parameter $q$ is a difficult problem in practical. For ill-posed problems that the large singular values dominated the solution, we generally select the parameter $q$ so that the last $n+1-q$ singular values of $C$ are very small. But, the smallest singular values of $C$ are rarely coincide \cite{SVHJV91}, it is realistic to define an error bound such that all singular values $\sigma_{i}$, satisfying $|\sigma_{i} - \sigma_{n+1}| < \epsilon$, are considered to coincide with $\sigma_{n+1}$.
Then, discarding the smallest singular values, and combining with the partition form \eqref{QI:TLS:Eq:PVq}, we can obtain the truncated total least squares (TTLS) solution \cite{SVHJV91}.

Due to the computational complexity, it is widely known that it may be unrealistic or extremely expensive to compute SVD for large-scale discrete ill-posed problems. Huffel \cite{SVH90} presents a partial SVD method based on Householder transformation or Lanczos bidiagonalization, see \cite{OS81,AB83} and the references therein. Gloub \emph{et al.} \cite{FGH97,GHO99} propose to use Tikhonov regularization to solve TLS.
In \cite{ZH22}, Zare and Hajarian propose an efficient Gauss-Newton algorithm for solving regularized total least squares problems. Recently, many types of randomized algorithm have been presented to calculate the low-rank approximation. The aim is to gain a projection by a random matrix \cite{MSM14,RT08,WLRT08} or random sampling \cite{AMT10,MRT11} with preconditioning \cite{CRT11,RR12}, see details in \cite{HMT11}. In \cite{XZ13,XZ15}, Xiang and Zou propose a randomized algorithm for solving the regularized LS solutions of large-scale discrete inverse problems. In \cite{WXZ16}, Wei, Xie and Zhang propose a regularization method, combining Tikhonov regularization in general form with the truncated generalized singular value decomposition (GSVD). Then the randomized algorithms are adopted to implement the truncation process. This randomized GSVD for the regularization of the large-scale ill-posed problems can achieve good accuracy with less computational time and memory requirement than the classical regularization methods. In \cite{JY18}, Jia and Yang propose a modified truncated randomized algorithms for the large-scale discrete ill-posed problems with general-form regularization.
Randomized algorithms are also utilized for the total least squares problems by Xie \emph{et al.} \cite{XXW19} and for core reduction problem by Zhang \emph{et al.} \cite{ZW20}.
All of these randomized algorithms can reduce the running time and still maintain good accuracy.

Based on the quantum simulation of resonant transitions, Wang and Xiang propose a quantum algorithm for total least squares problems in \cite{WX19}, and it can achieve at least polynomial speedup over the known classical algorithms. In terms of quantum algorithms, many quantum machine learning algorithms use quantum random access memory (QRAM) \cite{GLM08} as a tool to prepare quantum states. Due to the strict physical conditions required to maintain the coherence of quantum, the theoretical model of QRAM has not been well realized in practice. Compared with the quantum recommendation systems \cite{KP16},
Tang \cite{ET18} proposes a quantum-inspired classical algorithm for recommendation systems within logarithmic time by using the efficient low-rank approximation techniques of Frieze, Kannan and Vempala (FKV) algorithm \cite{FKV04}. Motivated by the dequantizing techniques, other papers are also proposed to deal with some low-rank matrix operations, such as matrix inversion, singular value transformation, non-negative matrix factorization, support vector machine, general minimum conical hull problems, principal component analysis, canonical correlation analysis, statistical leverage scores \cite{CGLLTW20,JGS19,CLS19,DBH21,DHLT20,ET182,KMKM21,ADBL20,ZX21}.
Dequantizing techniques in those algorithms involve two technologies, the Monte-Carlo singular value decomposition and sampling techniques, which could efficiently simulate some special operations on low-rank matrices.

Inspired by the dequantizing techniques, we propose a quantum-inspired algorithm for solving the TLS problem. First, the sample model and data structure are given in Section \ref{sec:QI:TLS:Sample:Model}, including vector and matrix sample model. Second, we present the quantum-inspired truncated total least squares (QiTTLS) algorithm to compute the minimum norm solution in Section \ref{sec:QI:TLS:Randomized:Algorithm}. Next, we demonstrate the algorithm analysis in Section \ref{sec:QI:TLS:Algorithm:Analysis}. Finally, we demonstrate the effectiveness behaviors of the proposed algorithms with the numerical examples from Hansen's Regularization Tools \cite{PCH07} in Section \ref{sec:QI:TLS:experiments}, and conclusions are made in Section \ref{sec:QI:TLS:conclusions}.

Throughout the paper, For a matrix $A \in \mathbb{R}^{m \times n}$, $A^T$ is the transpose of $A$, $\|A\|_{2}$ and ${\rm Tr}(A)$ denote the spectral norm and the trace of $A$, respectively. The symbol $A^{\dag}$ represents the Moore-Penrose inverse of $A$.
For a vector $v \in \mathbb{R}^{m}$, $\|v\|_{2}$ and $\|v\|_{\infty}$ also denote the Euclidean norm and the infinity norm, respectively. For any $1 \leq i \leq m$ and $1 \leq j \leq n$, denote the $i$th row of $A$ as $A_{i,:}$, and the $j$th column of $A$ as $A_{:,j}$. The $(i,j)$-entry of $A$ is denoted by $A_{i,j}$. Similarly, $v_{i}$ stands for the $i$th entry of a vector $v$. Let $I \in \mathbb{R}^{m \times m}$ be identity matrix.

\section{Sample model and data structure}\label{sec:QI:TLS:Sample:Model}
Throughout the quantum-inspired algorithm, the key is focusing on how to input the given matrix and vector. Obviously, it is not possible to load the entire matrix and vector into memory because it costs at least linear time. In this work, we assume that matrices and vectors are well prepared, and can be sampled according to some natural probability distributions, which are found in many applications of machine learning \cite{KP16,ET18,FKV04,CGLLTW20,JGS19}. Let us start from the sampling assumption of the matrix.

\begin{Assumption}\label{QI:TLS:asu1}
Given a matrix $M \in \mathbb{R}^{m \times n}$, the following conditions hold.
\begin{enumerate}[1)]
  \item We can sample a row index $i \in [m]$ of $M$, where the probability of row $i$ being chosen is
  \begin{equation}\label{QI:TLS:P}
    P_{i} = \frac{\|M_{i,:}\|^2_{2}}{||M||_F^2}.
  \end{equation}
  \item For all $i \in [m]$, we can sample an index $j \in [n]$ according to $\mathcal{D}_{M_{i,:}}$, i.e.,  the probability of $j$ being chosen is
  \begin{equation}\label{QI:TLS:DM}
    \mathcal{D}_{M_{i,:}}(j) = \frac{|M_{i,j}|^2}{||M_{i,:}||^2_{2}}.
  \end{equation}
\end{enumerate}
\end{Assumption}

Frieze \emph{et al.} \cite{FKV04} use these similar assumption to present a sublinear algorithm for seeking low-rank approximation. In \cite{ET18}, Tang gives a quantum-inspired classical algorithm for recommendation systems by using these similar assumption. As pointed out in \cite{CGLLTW20,CLS19,ET182}, there is a low-overhead data structure that satisfies the sampling assumption. We first describe the data structure for a vector, then for a matrix.

\begin{Lemma}(Vector sample model)(\cite{ET182})\label{QI:SLSMC:lemma:sample:V}
There exists a data structure storing a vector $v \in \mathbb{R}^{n}$ with $s$ nonzero entries in $O(s{\rm log}\,n)$ space, with the following properties:
\begin{enumerate}[a)]
  \item Querying and updating an entry of $v$ in $O({\rm log}\,n)$ time;
  \item Finding $||v||^2_{2}$ in $O(1)$ time;
  \item Sampling from $\mathcal{D}_{v}$ in $O({\rm log}\,n)$ time.
\end{enumerate}
\end{Lemma}
In \cite{ET18}, Tang gives a similar binary search tree (BST) diagram to analyze this data structure, as shown in the following figure.
\begin{figure}[H]
\centering
\begin{tikzpicture}[level distance=1.5cm,
  level 1/.style={sibling distance=3.5cm},
  level 2/.style={sibling distance=1.8cm}]
  \node {$||v||^2_{2}$}
    child {node {$v_1^2 + v_2^2$}
      child {node {$v_1^2$}
        child {node {${\rm sgn}(v_1)$}}
        }
      child {node {$v_2^2$}
        child {node {${\rm sgn}(v_2)$}}
        }
    }
    child {node {$v_3^2 + v_4^2$}
    child {node {$v_3^2$}
        child {node {${\rm sgn}(v_3)$}}
        }
      child {node {$v_4^2$}
        child {node {${\rm sgn}(v_4)$}}
        }
    };
\end{tikzpicture}
\caption{BST data structure for $v = (v_1, v_2, v_3, v_4)^{T} \in \mathbb{R}^{4}$.}
\end{figure}
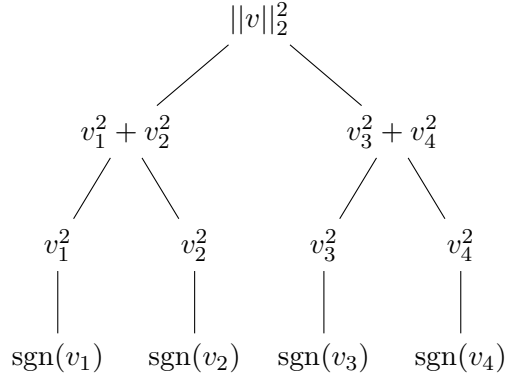

\begin{Lemma}(Matrix sample model)(\cite{ET182})\label{QI:TLS:lemma:sample:M}
Given a matrix $A \in \mathbb{R}^{m \times n}$ with $s$ nonzero entries in $O(s{\rm log}\,mn)$ space,
with the following properties:
\begin{enumerate}[a)]
  \item Querying and updating an entry of $A$ in $O({\rm log}\,mn)$ time;
  \item Sampling from $A_{i,:}$ for any $i \in [m]$ in $O({\rm log}\,n)$ time;
  \item Sampling from $A_{:,j}$ for any $j \in [n]$ in $O({\rm log}\,m)$ time;
  \item Finding  $||A||_F$, $\|A_{i,:}\|_{2}$ and $\|A_{:,j}\|_{2}$ in $O(1)$ time.
\end{enumerate}
\end{Lemma}
Follow the work in \cite{CGLLTW20}, here we present a similar BST diagram for a matrix $A \in \mathbb{R}^{4 \times 2}$ to analyze this data structure, as shown in the following figure.
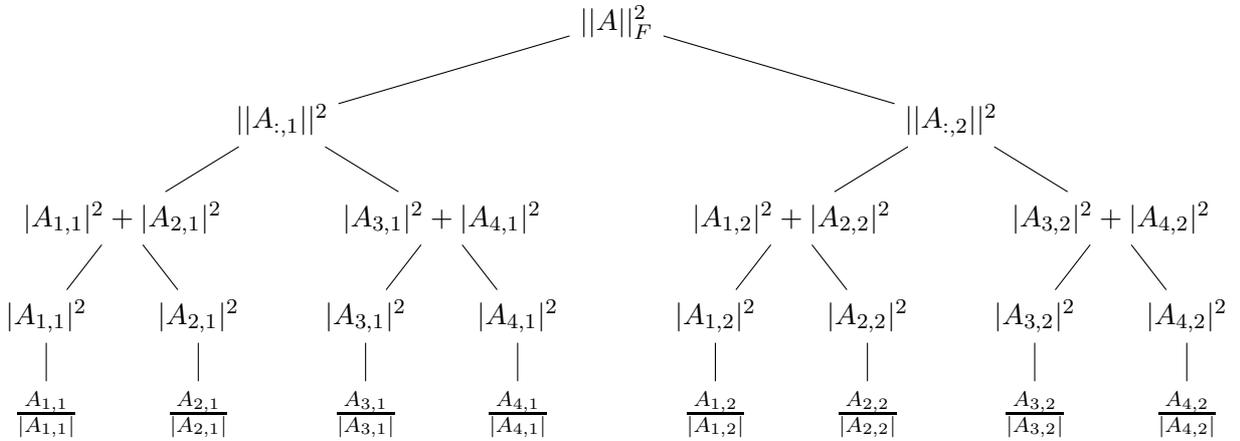
\begin{figure}[H]
\begin{tikzpicture}[level distance=1.3cm,
  level 1/.style={sibling distance=8.8cm},
  level 2/.style={sibling distance=4.2cm},
  level 3/.style={sibling distance=2cm}]
  \node {$||A||_F^2$}
    child {node {$||A_{:,1}||^2$}
            child {node {$|A_{1,1}|^2+|A_{2,1}|^2$}
                child {node {$|A_{1,1}|^2$}
                    child {node {$\frac{A_{1,1}}{|A_{1,1}|}$}}
                }
                child {node {$|A_{2,1}|^2$}
                    child {node {$\frac{A_{2,1}}{|A_{2,1}|}$}}
                }
        }
            child {node {$|A_{3,1}|^2+|A_{4,1}|^2$}
                child {node {$|A_{3,1}|^2$}
                    child {node {$\frac{A_{3,1}}{|A_{3,1}|}$}}
                }
                child {node {$|A_{4,1}|^2$}
                    child {node {$\frac{A_{4,1}}{|A_{4,1}|}$}}
                }
        }
    }
    child {node {$||A_{:,2}||^2$}
        child {node {$|A_{1,2}|^2+|A_{2,2}|^2$}
                child {node {$|A_{1,2}|^2$}
                    child {node {$\frac{A_{1,2}}{|A_{1,2}|}$}}
                }
                child {node {$|A_{2,2}|^2$}
                    child {node {$\frac{A_{2,2}}{|A_{2,2}|}$}}
                }
        }
        child {node {$|A_{3,2}|^2+|A_{4,2}|^2$}
                child {node {$|A_{3,2}|^2$}
                    child {node {$\frac{A_{3,2}}{|A_{3,2}|}$}}
                }
                child {node {$|A_{4,2}|^2$}
                    child {node {$\frac{A_{4,2}}{|A_{4,2}|}$}}
                }
        }
    };
\end{tikzpicture}
\caption{BST data structure for $A \in \mathbb{R}^{4 \times 2}$.}
\end{figure}

\section{Quantum-inspired algorithm}\label{sec:QI:TLS:Randomized:Algorithm}
There are already many efficient quantum-inspired classical algorithms, as discussed in the introduction. The main idea of quantum-inspired classical algorithms is follows. First, let us recall quantum algorithm, the key to quantum machine learning is quantum state preparation assumption. Given an input vector $v$, one can form a corresponding quantum state $|v\rangle$, QRAM provides support for many quantum machine learning (QML) algorithms from classical data to quantum data. Therefore, QML algorithms generally assume that some required quantum states are already prepared. In order to make the classical algorithm achieve nearly the same effect as the quantum algorithm. We need to use a similar data structure (defined in Section \ref{sec:QI:TLS:Sample:Model}) to satisfy state preparation assumption, which can also satisfy $\ell_2$-norm sampling assumption (defined in Section \ref{sec:QI:TLS:Sample:Model}). Second, for a given matrix $A \in \mathbb{R}^{m \times n}$, by the Eckhart-Young theorem in \cite{GL13}, the classical SVD can return a low-rank approximation $A_{(k)}$, which is the closest rank-$k$ matrix to $A$. However, the computational complexity is extremely expensive, which costs about $O(m^2 n)$.
Next, Frieze \emph{et al.} \cite{FKV04} propose a low approximation $A \widehat{V} \widehat{V}^{T}$ of $A$ in a linear time, which is an approximation projection onto the low-dimensional subspace spanned by $\widehat{V} \in \mathbb{R}^{m \times k}$, it needs about $O({\rm poly}(k,\frac{1}{\epsilon}, \log\,\frac{1}{\delta}))$, where $\epsilon$ and $\delta$ are error and failure probability, respectively. However, the right singular matrix $\widehat{V}$ is not column orthogonal matrix, it is not directly generalized to TLS. Drineas \emph{et al.} \cite{PDKM06} propose a linear time SVD algorithm to compute a low-rank matrix approximation, it outputs an approximation to the top $k$ singular values and the corresponding singular vectors. In fact, it is a modified FKV algorithm. Finally, combined the sample model with data structure and modified FKV algorithm, a quantum-inspired truncated total least squares (QiTTLS) algorithm is given. When the input matrix is low-rank, compared with the classical method, it can also achieve at least polynomial speedup.

\begin{algorithm} [H]
\caption{QiSVD: Quantum-inspired SVD algorithm.}\label{QI:TLS:QiTTLS:a1}
\begin{algorithmic}[1]
\REQUIRE The matrix $A \in \mathbb{R}^{m \times n}$ and vector $b$ satisfies the sample model and data structure,  $\epsilon > 0$, $k$ and $\delta \in (0,1)$. \\
\ENSURE The approximate left singular matrix $\hat{V}$.\\
\STATE Define $C = [A,\,b]$, $\xi = \frac{\epsilon}{2\epsilon + 4}$, $\alpha = \frac{\xi}{100 k^4}$, $\theta = \alpha \xi$ and $p = \left\lceil\frac{1}{\theta^2 \delta}\right\rceil$; \label{QI:TLS:QiTTLS:al:s0}
\STATE  Independently sample $p$ row indices $i_{1}, i_{2}, \cdots, i_{p}$ of $C \in \mathbb{R}^{m \times (n+1)}$ according to the probability distribution $\{P_{1}, P_{2}, \cdots, P_{m}\}$ defined in assumption \ref{QI:TLS:asu1}; \label{QI:TLS:QiTTLS:al:s1}
\STATE  Let $S \in \mathbb{R}^{p \times (n+1)}$ be the matrix formed by the normalized rows $\frac{{C}_{i_{t},:}}{\sqrt{pP_{i_{t}}}}$ for $t \in [p]$, i.e., $S_{t,:} = \frac{{C}_{i_{t},:}}{\sqrt{pP_{i_{t}}}}$; \label{QI:TLS:QiTTLS:al:s2}
\STATE Independently sample $p$ column indices $j_{1}, j_{2}, \cdots, j_{p}$ of $S \in \mathbb{R}^{p \times (n+1)}$ according to the probability distribution $\{P'_{1}, P'_{2}, \cdots, P'_{n+1}\}$, where $P'_{j} = \sum\limits_{t=1}^{p} \frac{\mathcal{D}_{C_{i_{t},:}}(j)}{p}$, and $i_{1}, i_{2}, \cdots, i_{p}$ are the row indices in step \ref{QI:TLS:QiTTLS:al:s1}; \label{QI:TLS:QiTTLS:al:s3}
\STATE Let $W \in \mathbb{R}^{p \times p}$ be the matrix formed by the normalized columns $\frac{S_{:,j_{t}}}{\sqrt{pP'_{j_{t}}}}$ for $t \in [p]$, i.e., $W_{:,t} = \frac{S_{:,j_{t}}}{\sqrt{pP'_{j_{t}}}}$; \label{QI:TLS:QiTTLS:al:s4}
\STATE Compute the SVD of $W$. Denoted here by $W = \sum^{p}_{t = 1} \bar{\sigma}_{t} \bar{u}_{t} \bar{v}^{T}_{t}$, where $\bar{\sigma}_{1} \geq \bar{\sigma}_{2} \geq \cdots \geq \bar{\sigma}_{p}$; \label{QI:TLS:QiTTLS:al:s5}
\STATE Let $l = {\rm min}\{k, {\rm max}\{t \in [p]: \bar{\sigma}_{t}^2 \geq \alpha \|W\|_F^2\}\}$; \label{QI:TLS:QiTTLS:al:s6}
\STATE Calculate the approximate right singular values $\hat{V} = S^{T} \bar{U} \bar{\Sigma}^{-1} \in \mathbb{R}^{(n+1) \times l}$, where $\bar{U} = (\bar{u}_{1}, \bar{u}_{2}, \cdots, \bar{u}_{l}) \in \mathbb{R}^{p \times l}$ and $\bar{\Sigma} = (\bar{\sigma}_{1}, \bar{\sigma}_{2}, \cdots, \bar{\sigma}_{l}) \in \mathbb{R}^{l \times l}$. \label{QI:TLS:QiTTLS:al:s7}
\end{algorithmic}
\end{algorithm}

\begin{remark}
The original quantum-inspired algorithms have been proposed in \cite{ET18,FKV04,CGLLTW20,JGS19,CLS19,ET182}. Here we modify the quantum-inspired algorithm, which is the step \ref{QI:TLS:QiTTLS:al:s6} of QiSVD algorithm. By properly setting the parameter $\alpha$, the computational complexity is independent of the input matrix condition number, the details can be seen in Section \ref{sec:QI:TLS:Algorithm:Analysis}. Moreover, the sampling order is modified here.
First, select $p$ rows of matrix $C$ according to a probability  $P_{i} = \frac{\|C_{i,:}\|_{2}}{\|C\|_F^2}$, and scale each row to form a matrix $S$. Second, select $p$ columns of matrix $S$ according to a probability $P'_{j} = \sum\limits_{t=1}^{p} \frac{\mathcal{D}_{C_{i_{t},:}}(j)}{p}$, and scale each column to form a matrix $W$. Finally, the singular values of matrix $W$ and the corresponding left singular vectors are calculated by the SVD of $W$, then form the approximate right singular vectors of matrix C.
\end{remark}

\begin{remark}
In QiSVD algorithm, the truncation parameter $k$ is previous determined, it is estimated by the randomized algorithm with GCV \cite{GHW21}. The parameter $p$ is used to adjust the balance between the computational complexity and reliability.
\end{remark}

\begin{Lemma}
Given a matrix $A \in \mathbb{R}^{m \times n}$ and a vector $b \in \mathbb{R}^{m}$ with $C = [A,\,b]$ satisfying the sample model and data structure. In QiSVD algorithm, $W$ can be formed by sampling $p$ rows of $C$ with probabilities $\{P_i\}_{i=1}^{m}$ and $p$ columns of $S$ with probabilities $\{P'_j\}_{j=1}^{n+1}$ where $P_{i} = \frac{\|C_{i,:}\|_{2}}{\|C\|_F^2}$ and $P'_{j} = \sum\limits_{t=1}^{p} \frac{\mathcal{D}_{C_{i_{t},:}}(j)}{p}$. Then, it holds that
\begin{equation}\label{QI:TLS:Eq:CFSFWF}
    \|C\|_F = \|S\|_F = \|W\|_F.
\end{equation}
\end{Lemma}

\begin{proof}
For any row $t \in [p]$ of matrix $S$, using Eq.\eqref{QI:TLS:P} we have
\begin{equation}\label{QI:TLS:Eq:SFAF}
    \|S\|_F^2 = \sum_{t=1}^{p}\|S_{t,:}\|^2_{2} = \sum_{t=1}^{p}\left\|\frac{C_{i_{t},:}}{\sqrt{pP_{i_{t}}}}\right\|^2_{2} = \sum_{t=1}^{p}\frac{\|C_{i_{t},:}\|^2_{2}}{p\frac{\|C_{i_{t},:}\|^2_{2}}{||C||_F^2}} = \sum_{t=1}^{p}\frac{\|C\|_F^2}{p} = \|C\|_F^2.
\end{equation}

Combining the definition $\mathcal{D}_{C_{i,:}}(j)$ in assumption \ref{QI:TLS:asu1}, and using Eq.\eqref{QI:TLS:Eq:SFAF}, it follows that
\begin{equation}\label{QI:TLS:Eq:qj}
    P'_{j} = \sum\limits_{t=1}^{p} \frac{\mathcal{D}_{C_{i_{t},:}}(j)}{p} = \frac{1}{\|C\|_F^2}\sum\limits_{t=1}^{p} \frac{|C_{i_t,j}|^2}{p\frac{\|C_{i_t,:}\|^2_{2}}{\|C\|_F^2}} = \frac{1}{\|C\|_F^2}\sum\limits_{t=1}^{p} \frac{|C_{i_t,j}|^2}{pP_{i_t}} = \frac{\|S_{:,j}\|^2_{2}}{\|S\|_F^2}.
\end{equation}

For any column $t \in [p]$ of matrix $W$, it yields that
\begin{equation}\label{QI:TLS:Eq:WFSF}
    \|W\|_F^2 = \sum_{t=1}^{p}\|W_{:,t}\|^2_{2} = \sum_{t=1}^{p}\left\|\frac{S_{:,j_{t}}}{\sqrt{p P'_{j_{t}}}}\right\|^2_{2} = \sum_{t=1}^{p}\frac{\|S_{:,j_{t}}\|^2_{2}}{p\frac{\|S_{:,j_{t}}\|^2_{2}}{\|S\|_F^2}} = \sum_{t=1}^{p}\frac{\|S\|_F^2}{p} = \|S\|_F^2.
\end{equation}
The results follows from Eqs.\eqref{QI:TLS:Eq:SFAF} and \eqref{QI:TLS:Eq:WFSF}.
\hfill
\end{proof}

\begin{Lemma}(\cite{FKV04})\label{QI:TLS:Lemma:MTMNTN}
Given a matrix $M \in \mathbb{R}^{m \times n}$, let $P = \{P_{1}, P_{2}, \cdots, P_{m}\}$ be a probability distribution on $[m]$ such that $P_{i} = \frac{\|M_{i,:}\|^2}{\|M\|^2_F}, i \in [m]$. Let $(i_1, i_2, \cdots, i_{p})$ be a sequence of $p$ independent samples from $[m]$, each chosen according to distribution $P$. Let $N$ be the ${p \times n}$ matrix with
\begin{equation}\label{QI:TLS:Eq:NN}
    N_{t,:} = \frac{M_{i_t,:}} {\sqrt{p P_{i_t}}}, \ \ t \in [p].
\end{equation}
Then, For all $\theta > 0$, it holds that
\begin{equation}\label{QI:TLS:Eq:Pmmnn}
    {\rm Pr} (\|M^{T}M-N^{T}N\|_F \geq \theta \|M\|_F^2 ) \leq \frac{1}{\theta^2 p}.
\end{equation}
\end{Lemma}

\begin{Lemma}(\cite{ZX21})\label{QI:TLS:Lemma:MMTNNT}
Given a matrix $M \in \mathbb{R}^{m \times n}$, let $P' = \{P'_{1}, P'_{2}, \cdots, P'_{n}\}$ be a probability distribution distribution on $[n]$ such that $P'_{j} = \frac{\|M_{:,j}\|^2}{\|M\|^2_F}, j \in [n]$. Let $(j_1, j_2, \cdots, j_{p})$ be a sequence of $p$ independent samples from $[n]$, each chosen according to distribution $P'$. Let $N$ be the ${m \times p}$ matrix with
\begin{equation}\label{QI:TLS:Eq:NN}
    N_{:,t} = \frac{M_{:,j_t}}  {\sqrt{p P'_{j_t}}}, \ \ t \in [p].
\end{equation}
Then, for all $\theta > 0$, it yields that
\begin{equation}\label{QI:TLS:Eq:Pmmnn}
    {\rm Pr} (\|MM^{T} - NN^{T}\|_F \geq \theta \|M\|_F^2 ) \leq \frac{1}{\theta^2 p}.
\end{equation}
\end{Lemma}

To summarize of QiSVD algorithm, according to lemmas \ref{QI:TLS:Lemma:MTMNTN} and \ref{QI:TLS:Lemma:MMTNNT}, with probability at least $1 - \delta$, we have $\|C^{T}C - S^{T}S\|_{F} \leq \theta \|C\|^2_{F}$ and $\|SS^{T} - WW^{T} \|_{F} \leq \theta \|S\|^2_{F}$. That is, we can gain the approximate right singular matrix $\hat{V}$ of input matrix $C$, i.e., $\hat{V} = S^{T} \bar{U} \bar{\Sigma}^{-1} \in \mathbb{R}^{(n+1) \times l}$. Once obtained the approximate right singular matrix $\hat{V}$ of matrix $C$, then we can make a similar partition of $\hat{V}$. As a result, we can form the so-called TTLS solution. The detailed algorithm analysis is given as follows.

\begin{algorithm} [H]
\caption{QiTTLS: Quantum-inspired algorithm for TTLS.}\label{QI:TLS:QiTTLS:a2}
\begin{algorithmic}[1]
\REQUIRE The matrix $A \in \mathbb{R}^{m \times n}$ and vector $b$ satisfies the sample model and data structure, 
$\epsilon > 0$, $k$, $\delta \in (0,1)$, the parameter $l$ is defined in QiSVD algorithm, and $d  \leq l$.\\
\ENSURE $x_{\rm QiTTLS}$.\\
\STATE Compute $\hat{V} = {\rm QiSVD}(C)$, with the parameters $\epsilon$, $k$, $\delta$ and $C = [A,\,b]$; \label{QI:TLS:MQiTTLS:a2:s1}
\STATE Let $\hat{\mathbf{V}}_{11} = \hat{V}_{1:n,1:d}$, $\hat{\mathbf{v}}_{21}= \hat{V}_{n+1,1:d}$, and form the solution
$x_{\rm QiTTLS} = (\hat{\mathbf{V}}_{11}^T)^{\dag}\hat{\mathbf{v}}_{21}^{T}$.\label{QI:TLS:MQiTTLS:a2:s2}
\end{algorithmic}
\end{algorithm}

\begin{remark}
In most cases, quantum-inspired algorithm can estimate accurately the large singular values of large-scale matrices. Based on the above QiSVD algorithm, we can get a good approximation of the right singular vectors associated with the large singular values. Therefore, we choose the solution of TTLS as $x_{\rm QiTTLS} = (\hat{\mathbf{V}}_{11}^T)^{\dag}\hat{\mathbf{v}}_{21}^{T}$.
\end{remark}

\begin{algorithm} [H]
\caption{TTLS: Classical TTLS.}\label{QI:TLS:TTLS:a3}
\begin{algorithmic}[1]
\REQUIRE $A \in \mathbb{R}^{m \times n}$, $b \in \mathbb{R}^{m}$, $d \leq q$.\\
\ENSURE $x_{\rm TTLS}$.\\
\STATE Compute $U \Sigma V^{T} = {\rm SVD}(C)$, where $C = [A,\,b]$; \label{QI:TLS:TTLS:a3:s1}
\STATE Let $\mathbf{V}_{11} = V_{1:n,1:d}$, $\mathbf{v}_{21}= V_{n+1,1:d}$, and form the solution
$x_{\rm TTLS} = (\mathbf{V}_{11}^T)^{\dag} \mathbf{v}_{21}^{T}$.\label{QI:TLS:TTLS:a3:s2}
\end{algorithmic}
\end{algorithm}

\begin{remark}
For most ill-conditioned matrices in practical applications, the parameter $q$ in condition \eqref{QI:TLS:Eq:sqsn} is generally difficult to know in advance. Here we choose the truncation parameter
$d \leq q$, and other small singular values can be regarded as zero. Then, the relative error between the solution $x_{\rm QiTTLS}$ of the QiTTLS algorithm and the solution $x_{\rm TTLS}$ of the TTLS algorithm is analyzed. The proof will be given in theorem \ref{QI:TLS:Theorem:UI}.
\end{remark}

\section{Algorithm analysis}\label{sec:QI:TLS:Algorithm:Analysis}
\subsection{Error estimates}\label{sec:QI:TLS:Error:Alys}
\begin{Lemma}(Weyl's inequality \cite{HW12})\label{QI:SLSMC:Lemma:SmSn}
For two matrices $M \in \mathbb{C}^{m \times n}$, $N \in \mathbb{C}^{m \times n}$ and any $i\in [{\rm min}(m,n)]$, $|\sigma_{i}(M) - \sigma_{i}(N)| \leq \|M - N\|_{2}$.
\end{Lemma}

\begin{Lemma}\label{QI:TLS:Lemma:V2TVF}
Given a matrix $A \in \mathbb{R}^{m \times n}$ and a vector $b \in \mathbb{R}^{m}$ satisfying the sample model and data structure, the parameters $(\epsilon, \delta, k)$ in the specified range of QiSVD algorithm. QiSVD algorithm outputs the approximate right singular matrix $\hat{V} \in \mathbb{R}^{(n+1) \times l}$, then with probability $1- \delta$, it holds that
\begin{equation}\label{QI:SLSMC:Eq:VFVR}
    \begin{aligned}
        \|\hat{V}^{T} \hat{V} \|_{2} & \leq 1 + \xi, \\
        \|\hat{V}\|^2_F & \leq k +  \sqrt{k} \xi, \\
        \|\hat{V}\|_{2} &  \leq \sqrt{1 + \xi}.
    \end{aligned}
\end{equation}
\end{Lemma}

\begin{proof}
According to $\hat{V} = S^{T} \bar{U} \bar{\Sigma}^{-1}$ and $\bar{\sigma}^2_{l} \geq \alpha \|W\|_F^2$ in step \ref{QI:TLS:QiTTLS:al:s6} of QiSVD algorithm, since $\theta = \alpha \xi$ in QiSVD algorithm, based on lemma \ref{QI:TLS:Lemma:MTMNTN}, then with probability $1- \delta$, we have
\begin{equation}\label{QI:TLS:Eq:VTVI2}
  \begin{aligned}
    \|\hat{V}^{T} \hat{V} - I\|_{2}
     & = \|\bar{\Sigma}^{-T} \bar{U}^{T} SS^{T} \bar{U}\bar{\Sigma}^{-1} - \bar{\Sigma}^{-T} \bar{U}^{T}WW^{T}\bar{U}\bar{\Sigma}^{-1}\|_{2} \\
     & = \|\bar{\Sigma}^{-T} \bar{U}^{T}(SS^{T} - WW^{T})\bar{U}\bar{\Sigma}^{-1} \|_{2} \\
     & \leq \|\bar{\Sigma}^{-1}\|^2_{2} \| \bar{U} \|^2_{2} \|SS^{T} - WW^{T}\|_{2} \\
     & \leq \|\bar{\Sigma}^{-1}\|^2_{2}  \|\bar{U} \|^2_{2} \|SS^{T} - WW^{T}\|_F \\
     & \leq \frac{\|SS^{T} - WW^{T}\|_F}{\alpha\|W\|^2_F} \\
     & \leq \frac{\theta \|S\|^2_F}{\alpha\|W\|^2_F} \\
     & = \xi.
  \end{aligned}
\end{equation}

Similar to Eq.\eqref{QI:TLS:Eq:VTVI2}, with probability at least $1 - \delta$, it yields that
\begin{equation}\label{QI:TLS:Eq:VjIVF}
\begin{aligned}
    \|\hat{V}^{T} \hat{V} - I\|_F
     & = \|\bar{\Sigma}^{-T} \bar{U}^{T} SS^{T} \bar{U}\bar{\Sigma}^{-1} - \bar{\Sigma}^{-T} \bar{U}^{T}WW^{T}\bar{U}\bar{\Sigma}^{-1}\|_F \\
     & = \|\bar{\Sigma}^{-T} \bar{U}^{T}(SS^{T} - WW^{T})\bar{U}\bar{\Sigma}^{-1} \|_F \\
     & \leq \|\bar{\Sigma}^{-1}\|^2_{2} \| \bar{U} \|^2_{2} \|SS^{T} - WW^{T}\|_F \\
     & \leq \frac{\|SS^{T} - WW^{T}\|_F}{\alpha\|W\|^2_F} \\
     & \leq \frac{\theta \|S\|^2_F}{\alpha\|W\|^2_F} \\
     & = \xi.
\end{aligned}
\end{equation}

Set $\hat{V}^{T} \hat{V} - I = \hat{E}$ with $\|\hat{E}\|_2 \leq \xi$, since $l \leq k$, we have
\begin{equation}\label{QI:TLS:Eq:VRV}
    \begin{aligned}
    \|\hat{V}^{T} \hat{V}\|_{2}  = \|I + \hat{V}^{T} \hat{V} - I\|_{2} \leq \| I \|_{2} + \|\hat{V}^{T} \hat{V} - I \|_{2}  \leq 1 + \xi,\\
    \|\hat{V}\|_F^2 = \sum_{j=1}^{l}  (\hat{V}_{:,j})^{T} \hat{V}_{:,j} = \sum_{j=1}^{l}(1 + \hat{E}_{jj}) = l +  \sum_{j=1}^{l}\hat{E}_{jj} \leq l + \sqrt{l} \xi \leq k + \sqrt{k} \xi.
    \end{aligned}
\end{equation}

By lemma \ref{QI:SLSMC:Lemma:SmSn} and Eq.\eqref{QI:TLS:Eq:VTVI2}, for any $i \in [l]$, it holds that
\begin{equation}\label{QI:TLS:Eq:SigmaV}
    |\sigma_{i}(\hat{V}^{T}\hat{V}) - \sigma_{i}(I)| \leq  \|\hat{V}^{T} \hat{V} - I\|_{2} \leq \xi,
\end{equation}
which implies that $\sigma_{\rm max} (\hat{V}^{T}\hat{V}) \leq 1+ \xi$. Moreover, we have
\begin{equation}\label{QI:TLS:Eq:normV2}
    \|\hat{V}\|_{2}  \leq \sqrt{1 + \xi}.
\end{equation}
\end{proof}

\begin{Lemma}\label{QI:TLS:Lemma:VFVR}
Given a matrix $A \in \mathbb{R}^{m \times n}$ and a vector $b \in \mathbb{R}^{m}$ satisfying the sample model and data structure, the parameters $(\epsilon, \delta, k)$ in the specified range of QiSVD algorithm, the matrices $S$, $W$ and the singular values $\bar{\sigma}_{t}$, $t \in [p]$ of matrix $W$ are defined in QiSVD algorithm. QiSVD algorithm outputs the approximate right singular matrix $\hat{V}  \in \mathbb{R}^{(n+1) \times l}$, then with probability $1- \delta$, it holds that
\begin{equation}\label{QI:SLSMC:Eq:VFVR}
    \begin{aligned}
\|C \hat{V}\|_F^2
\geq \sum^{l}_{t=1}\bar{\sigma}^2_{t} - \frac{2 \theta \|S \|^2_F}{\sqrt{\alpha}} - \theta(k + \sqrt{k} \xi) \|C\|^2_F.
    \end{aligned}
\end{equation}
where $C = [A,\,b]$.
\end{Lemma}

\begin{proof}
By a direct computation, applying Eq.\eqref{QI:TLS:Eq:VRV}, then with probability $1- \delta$, it follows that
\begin{equation}\label{QI:TLS:Eq:VVE6}
    \begin{aligned}
         \left|{\rm Tr} \left[\hat{V}^{T}(C^{T}C - S^{T}S)\hat{V} \right]\right|
        & = \left|\sum_{i=1}^{l} (\hat{V}^{T})_{i,:} (C^{T}C - S^{T}S)  \hat{V}_{:,i}  \right| \\
        & \leq \sum_{i=1}^{l} \left| (\hat{V}^{T})_{i,:} (C^{T}C - S^{T}S) \hat{V}_{:,i} \right| \\
        & \leq \|C^{T}C - S^{T}S\|_{2} \|\hat{V}\|_F^2 \\
        & \leq \|C^{T}C - S^{T}S\|_F \|\hat{V}\|_F^2 \\
        & \leq \theta (k + \sqrt{k} \xi) \|C\|^2_F.
    \end{aligned}
\end{equation}
That is, we have
\begin{equation}\label{QI:TLS:Eq:CVR}
\|C \hat{V}\|_F^2
=  {\rm Tr} \left( \hat{V}^{T} S^{T} S \hat{V} \right) + {\rm Tr} \left[ \hat{V}^{T}  (C^{T}C - S^{T}S)\hat{V} \right]
\geq \|S\hat{V}\|^2_F - \theta(k + \sqrt{k} \xi) \|C\|^2_F.
\end{equation}
According to the triangle inequality, and using $\hat{V} = S^{T} \bar{U} \bar{\Sigma}^{-1}$, we have
\begin{equation}\label{QI:TLS:Eq:WWUS}
    \begin{aligned}
         \left|\| WW^{T} \bar{U} \bar{\Sigma}^{-1}\|_F - \|( SS^{T} - WW^{T}) \bar{U} \bar{\Sigma}^{-1}\|_F\right|
         & \leq \| WW^{T} \bar{U} \bar{\Sigma}^{-1} + ( SS^{T} - WW^{T}) \bar{U} \bar{\Sigma}^{-1} \|_F
         = \|S \hat{V}\|_F,
    \end{aligned}
\end{equation}
which implies that
\begin{equation}\label{QI:TLS:Eq:SVR}
    \begin{aligned}
         \|S \hat{V}\|_F^2
         & \geq (\| WW^{T} \bar{U} \bar{\Sigma}^{-1}\|_F - \|( SS^{T} - WW^{T}) \bar{U} \bar{\Sigma}^{-1}\|_F)^2 \\
         & = \left[\left(\sum^{l}_{t=1}\bar{\sigma}^2_{t}\right)^{\frac{1}{2}} - \|(SS^{T} - WW^{T}) \bar{U} \bar{\Sigma}^{-1}\|_F\right]^2 \\
         & \geq \left[\left(\sum^{l}_{t=1}\bar{\sigma}^2_{t}\right)^{\frac{1}{2}}- \frac{\theta \|S \|^2_F}{\sqrt{\alpha} \|W\|_F}\right]^2 \\
         & > \sum^{l}_{t=1}\bar{\sigma}^2_{t} - \frac{2 \theta \|S \|^2_F}{\sqrt{\alpha}},
    \end{aligned}
\end{equation}
where the results follows from Eqs.\eqref{QI:TLS:Eq:CVR}, \eqref{QI:TLS:Eq:WWUS} and \eqref{QI:TLS:Eq:SVR}.
\end{proof}

\begin{Lemma}\label{QI:TLS:Lemma:SSWCE}
Given a matrix $A \in \mathbb{R}^{m \times n}$ and a vector $b \in \mathbb{R}^{m}$ satisfying the sample model and data structure, the parameters $(\epsilon, \delta, k)$ in the specified range of QiSVD algorithm, the matrices $S$, $W$ and the singular values $\bar{\sigma}_{t}, t \in [p]$ of matrix $W$ are defined in QiSVD algorithm. Suppose that $C = [A,\,b]$ has singular values $\sigma_{1}, \sigma_{2}, \cdots, \sigma_{n + 1}$. QiSVD algorithm outputs the approximate right singular matrix $\hat{V}$, then with probability $1- \delta$, it holds that
\begin{equation}\label{QI:TLS:Eq:SSWCE}
\sum_{t=1}^{k}\bar{\sigma}^2_{t}  \geq \sum_{t=1}^{k}\sigma^2_{t} - 2\sqrt{k} \theta \|S\|^2_F.
\end{equation}
\end{Lemma}

\begin{proof}
Based on Cauchy-Schwarz inequality and Hoffman-Wielandt theorem \cite{GL13}, and suppose matrix $S$ has singular values $\tilde{\tau}_{1}, \tilde{\tau}_{2}, \cdots, \tilde{\tau}_{p}$, then by lemma \ref{QI:TLS:Lemma:MTMNTN}, we have
\begin{equation}\label{QI:TLS:Eq:SCE}
    \begin{aligned}
\left|\sum_{t=1}^{k}(\tilde{\tau}^2_{t} - \sigma^2_{t})\right|
& \leq \sqrt{k} \left[\sum_{t=1}^{k} (\tilde{\tau}^2_{t} - \sigma^2_{t})^2\right]^{\frac{1}{2}} \\
& = \sqrt{k} \left[\sum_{t=1}^{k} (\sigma_{t}(S^{T}S) - \sigma_{t}(C^{T}C))^2\right]^{\frac{1}{2}} \\
& \leq \sqrt{k} \|S^{T}S - C^{T}C\|_F \\
& \leq \sqrt{k} \theta \|S\|^2_F.
    \end{aligned}
\end{equation}
where $\sigma_{j}(A)$ denote the $j$th singular value of $A$.

Similarly, by lemma \ref{QI:TLS:Lemma:MMTNNT}, it yields that
\begin{equation}\label{QI:TLS:Eq:WSE}
    \begin{aligned}
\left|\sum_{t=1}^{k} (\bar{\sigma}^2_{t} - \tilde{\tau}^2_{t})\right|
& \leq \sqrt{k} \left[\sum_{t=1}^{k} (\bar{\sigma}^2_{t} - \tilde{\tau}^2_{t})^2\right]^{\frac{1}{2}} \\
& = \sqrt{k} \left[\sum_{t=1}^{k} (\sigma_{t}(WW^{T}) - \sigma_{t}(SS^{T}))^2\right]^{\frac{1}{2}}\\
& \leq \sqrt{k} \|SS^{T} - WW^{T}\|_F  \\
& \leq \sqrt{k} \theta \|S\|^2_F.
    \end{aligned}
\end{equation}
Combining Eqs.\eqref{QI:TLS:Eq:SCE} with \eqref{QI:TLS:Eq:WSE}, it follows that
\begin{equation}\label{QI:TLS:Eq:WCE}
\begin{aligned}
\left|\sum_{t=1}^{k}  (\bar{\sigma}^2_{t} - \sigma^2_{t})\right|
& = \left|\sum_{t=1}^{k}  (\bar{\sigma}^2_{t} - \tilde{\tau}^2_{t}  + \tilde{\tau}^2_{t}  - \sigma^2_{t}) \right| \\
& \leq \left|\sum_{t=1}^{k} (\bar{\sigma}^2_{t} - \tilde{\tau}^2_{t})\right|  + \left|\sum_{t=1}^{k} (\tilde{\tau}^2_{t}  - \sigma^2_{t})\right| \\
& \leq \sqrt{k} \|SS^{T} - WW^{T}\|_F + \sqrt{k} \|S^{T}S - C^{T}C\|_F \\
& = 2\sqrt{k} \theta \|S\|^2_F,
\end{aligned}
\end{equation}
which implies that
\begin{equation}\label{QI:TLS:Eq:SWCE}
\sum_{t=1}^{k}\bar{\sigma}^2_{t} \geq \sum_{t=1}^{k}\sigma^2_{t} - 2\sqrt{k} \theta \|S\|^2_F.
\end{equation}
\end{proof}

\begin{Lemma}\label{QI:TLS:Lemma:CVVFCKE}
Suppose that a matrix $A \in \mathbb{R}^{m \times n}$ and a vector $b \in \mathbb{R}^{m}$ satisfying the sample model and data structure, the parameters $(\epsilon, \delta, k)$ in the specified range of QiSVD algorithm. QiSVD algorithm outputs the approximate right singular matrix $\hat{V} \in \mathbb{R}^{(n+1) \times l}$, then with probability $1- \delta$, it follows that

\begin{equation*}
    \|C - C\hat{V}\hat{V}^{T}\|^2_F <  \|C - C_{(k)}\|^2_F + \epsilon \|C\|_F^2,
\end{equation*}
where $C = [A,\,b]$ and $C_{(k)} = \sum_{i= 1}^{k} \sigma_{i} u_{i} v_{i}^{T}$.
\end{Lemma}

\begin{proof}
Motivated by the work in \cite{PDKM06}, using the thin singular value decomposition of $\hat{V} \in \mathbb{R}^{(n+1) \times l}$, namely
\begin{equation}\label{QI:TLS:Eq:SVDV}
    \hat{V} = X \Sigma_{\hat{V}} Y^{T},
\end{equation}
where $X \in \mathbb{R}^{(n+1) \times l}$, $Y \in \mathbb{R}^{l \times l}$ are orthonormal and $\Sigma_{\hat{U}} \in \mathbb{R}^{l \times l}$ is a diagonal matrix, we obtain
\begin{equation}\label{QI:TLS:Eq:VVE2}
\|C \hat{V}\|_F^2  = \|CX \Sigma_{\hat{V}}\|_F^2 \leq \|CX\|_F^2 \|\Sigma_{\hat{V}}\|^2_{2} = \|CX\|_F^2 \|\hat{V}^{T}\hat{V}\|_{2}.
\end{equation}
Since $\xi = \frac{\epsilon}{2\epsilon+4} < 1$, using lemma \ref{QI:TLS:Lemma:V2TVF}, it holds that
\begin{equation}\label{QI:TLS:Eq:CXE}
        \|CX\|_F^2 \geq \frac{1}{\|\hat{V}^{T}\hat{V}\|_{2}} \|C \hat{V}\|_F^2 \geq \frac{1}{1 + \xi}  \|C \hat{V}\|_F^2 \geq ({1 - \xi}) \|C \hat{V}\|_F^2.
\end{equation}
Based on lemma \ref{QI:TLS:Lemma:VFVR}, we have
\begin{equation}\label{QI:TLS:Eq:FCXE}
\|C X \|_F^2
\geq ({1 - \xi})  \left[ \sum^{l}_{t=1}\bar{\sigma}^2_{t} - \frac{2 \theta \|S \|^2_F}{\sqrt{\alpha}} - \theta(k + \sqrt{k} \xi) \|C\|^2_F \right].
\end{equation}

Since $l = {\rm min}\{k, {\rm max}\{t \in [p]: \bar{\sigma}_{t}^2 \geq \alpha \|W\|_F^2\}\}$ in the step \ref{QI:TLS:QiTTLS:al:s6} of QiSVD algorithm, set $t_{\rm max} = {\rm max}\{t \in [p]: \bar{\sigma}_{t}^2 \geq \alpha \|W\|_F^2\}$, there are two items as follows.
\begin{itemize}
  \item $l = t_{\rm max} \leq k$\\
    Since $l = t_{\rm max} \leq k$, we have $\bar{\sigma}_{t}^2 < \alpha \|W\|_F^2$, $t = l+1, l+2, \cdots, k$, i.e.,
\begin{equation}\label{QI:TLS:Eq:tktw}
    \sum_{t=l+1}^{k}\bar{\sigma}^2_{t} < (k - l)\alpha\|W\|_F^2.
\end{equation}
Using lemma \ref{QI:TLS:Lemma:SSWCE} and  Eq.\eqref{QI:TLS:Eq:tktw}, it yields that
\begin{equation}\label{QI:TLS:Eq:sskl}
\sum_{t=1}^{l}\bar{\sigma}^2_{t}
= \sum_{t=1}^{k}\bar{\sigma}^2_{t} - \sum_{t = l+1}^{k}\bar{\sigma}^2_{t} \\
> \sum_{t=1}^{k}\sigma^2_{t} - 2\sqrt{k} \theta \|S\|^2_F - (k - l) \alpha \|W\|_F^2.
\end{equation}
  \item $l = k < t_{\rm max}$\\
Applying $l = k$ and lemma \ref{QI:TLS:Lemma:SSWCE}, we obtain
\begin{equation}\label{QI:TLS:Eq:ltk}
\sum_{t=1}^{l}\bar{\sigma}^2_{t} = \sum_{t=1}^{k}\bar{\sigma}^2_{t} > \sum_{t=1}^{k}\sigma^2_{t} - 2\sqrt{k} \theta \|S\|^2_F.
\end{equation}
\end{itemize}

Using Eqs.\eqref{QI:TLS:Eq:sskl} and \eqref{QI:TLS:Eq:ltk}, the Eq.\eqref{QI:TLS:Eq:FCXE} becomes
\begin{equation}\label{QI:TLS:Eq:FCXE}
\begin{aligned}
\|C X \|_F^2
        & \geq ({1 - \xi}) \left[\sum_{t=1}^{k}\sigma^2_{t} - 2\sqrt{k} \theta \|S\|^2_F - (k - l) \alpha \|W\|_F^2 - \frac{2 \theta \|S \|^2_F}{\sqrt{\alpha}} - \theta(k + \sqrt{k} \xi) \|C\|^2_F \right] \\
        & \geq \sum_{t=1}^{k}\sigma^2_{t} - 2\sqrt{k} \theta \|C\|^2_F - k\alpha \|W\|_F^2 - \frac{2\theta\|S\|_F^2}{\sqrt{\alpha}} - \theta  (k + \sqrt{k} \xi) \|C\|^2_F - \xi \|W\|_F^2 \\
        & \geq \sum_{t=1}^{k}\sigma^2_{t} - \left[ 2\sqrt{k} \theta + k\alpha + 2\theta/\sqrt{\alpha} + \theta  (k + \sqrt{k} \xi) + \xi \right] \|C\|^2_F.
\end{aligned}
\end{equation}

Since $\alpha = \frac{\xi}{16 k^4 }$, $\theta = \alpha \xi$ and $\xi = \frac{\epsilon}{2\epsilon +4}$, it holds that
\begin{equation}\label{QI:TLS:Eq:klbesk}
\begin{aligned}
&~~~~ 2\sqrt{k} \theta + k\alpha + 2\theta/\sqrt{\alpha} + \theta  (k + \sqrt{k} \xi) + \xi \\
& = 2\sqrt{k} \alpha \xi + k\alpha + 2\sqrt{\alpha}\xi + k \alpha \xi + \sqrt{k} \alpha \xi^2 + \xi \\
& \leq 2k \alpha \xi + k\alpha + 2 k \sqrt{\alpha}\xi + k \alpha \xi + k \alpha \xi + \xi \\
& = 4k \alpha \xi + k\alpha + 2 k \sqrt{\alpha}\xi + \xi \\
& < \frac{\xi}{4 k} + \frac{\xi}{16 k} + \frac{\xi}{2k}  + \xi
< 2\xi.
\end{aligned}
\end{equation}
Combining Eqs.\eqref{QI:TLS:Eq:FCXE} and \eqref{QI:TLS:Eq:klbesk}, we have
\begin{equation}\label{QI:TLS:Eq:CXt2e}
\|C X \|_F^2
> \sum_{t=1}^{k}\sigma^2_{t} - 2\xi \|C\|^2_F.
\end{equation}

By direct computation, applying Eq.\eqref{QI:TLS:Eq:CXt2e}, we obtain
\begin{equation}\label{QI:TLS:Eq:CCXXT}
  \begin{aligned}
    \|C - C XX^{T}\|_F^2
    & = \|C\|_F^2 - \|CX\|_F^2 \\
    & < \|C\|_F^2 - \sum_{t=1}^{k}\sigma^2_{t} + 2\xi \|C\|^2_F\\
    & = \|C - C_{(k)}\|_F^2 + 2\xi \|C\|^2_F,
  \end{aligned}
\end{equation}
where $C_{(k)} = \sum_{i= 1}^{k} \sigma_{i} u_{i} v_{i}^{T}$.

Using Eqs.\eqref{QI:TLS:Eq:SVDV} and \eqref{QI:TLS:Eq:VjIVF}, we have
\begin{equation}\label{QI:TLS:Eq:XXVVR}
    \|XX^{T} - \hat{V}  \hat{V}^{T} \|_F
    = \|X(I - \Sigma_{\hat{V}}^{2}) X^{T}\|_F   \\
    = \|I - \Sigma_{\hat{V}}^{2}\|_F \\
    = \|Y(I - \Sigma_{\hat{V}}^{2})Y^{T}\|_F \\
    = \|I - \hat{V}^{T} \hat{V} \|_F \\
    \leq \xi.
\end{equation}

By the triangle inequality, combining Eqs.\eqref{QI:TLS:Eq:CCXXT} with \eqref{QI:TLS:Eq:XXVVR}, for $\xi \geq 0$, it yields that
\begin{equation}\label{QI:TLS:Eq:CCVVR}
\begin{aligned}
   \|C - C\hat{V}  \hat{V}^{T} \|_F^2
   & \leq \left(\|C - CXX^{T}\|_F + \|CXX^{T} - C\hat{V}  \hat{V}^{T}\|_F\right)^2 \\
   & \leq (1 + \xi)\|C - CXX^{T}\|_F^2 + (1 + 1/\xi)\|CXX^{T} - C\hat{V}  \hat{V}^{T}\|_F^2  \\
   & \leq (1 + \xi)\|C - CXX^{T}\|_F^2 + (1 + 1/\xi)\|XX^{T} - \hat{V}  \hat{V}^{T} \|_F^2 \|C\|_F^2 \\
   & < (1 + \xi)  (\|C - C_{(k)}\|_F^2 + 2\xi \|C\|_F^2) + (1 + 1/\xi) \xi^2 \|C\|_F^2 \\
   & = \|C - C_{(k)}\|_F^2 + (3\xi^2 + 4 \xi) \|C\|_F^2 \\
   & < \|C - C_{(k)}\|_F^2 + \epsilon \|C\|_F^2,
\end{aligned}
\end{equation}
where the last equality follows that $3\xi^2 + 4 \xi = 3 \left(\frac{\epsilon}{2\epsilon+4}\right)^2 + 4\left(\frac{\epsilon}{2\epsilon+4}\right) < \epsilon$.
\end{proof}

\begin{Lemma}(\cite{SJG91})\label{QI:TLS:Lemma:ERI}
Let $A$ be an $n \times n$ positive definite Hermitian matrix and $A = LL^{T}$ its Cholesky factorization. If $E$ is an $n \times n$ Hermitian matrix satisfying $\|A^{-1}\|_2 \|E\|_2 < 1$, then there is a unique Cholesky factorization
$$A + E = (L + G)(L + G)^{T},$$
and
$$\frac{\|G\|_F}{\|L\|_2} \leq \frac{\kappa} {\sqrt{2(1 - \|A^{-1}\|_2 \|E\|_2)}} \frac{\|E\|_F}{\|A\|_2}, $$
where $\kappa$ denotes the condition number of $A$.
\end{Lemma}

\begin{Lemma}\label{QI:TLS:Theorem:VZ}
Given a matrix $A \in \mathbb{R}^{m \times n}$ and a vector $b \in \mathbb{R}^{m}$ satisfying the sample model and data structure, the parameters $(\epsilon, \delta, k)$ in the specified range of QiSVD algorithm. QiSVD algorithm outputs the approximate right singular matrix $\hat{V}$, then with probability $1- \delta$, there exists a column orthonormal matrix $Z \in \mathbb{R}^{(n+1) \times l}$ such that
\begin{equation}\label{QI:TLS:Eq:U10}
    \|\hat{V} - Z\|_F  < \xi.
\end{equation}
\end{Lemma}

\begin{proof}
Using the QR decomposition of $\hat{V}$, i.e., $\hat{V} = Q\left[\begin{array}{c}
R \\
\mathbf{0}
\end{array}\right] \in \mathbb{R}^{(n+1) \times l}$, where $Q \in \mathbb{R}^{(n+1) \times (n+1)}$ be an orthonormal matrix and $R \in \mathbb{R}^{l \times l}$ be an upper triangular matrix. Then, applying Eq.\eqref{QI:TLS:Eq:VjIVF}, we have

\begin{equation}\label{QI:TLS:Eq:UURR}
    \|\hat{V}^{T}\hat{V} - I\|_F = \|R^{T}R - I\|_F \leq \xi.
\end{equation}
We find that $R^{T}R$ can be viewed as an approximate Cholesky decomposition of $I$.
That is, $R^{T}R = \hat{V}^{T}\hat{V} = I + \hat{E}$, where $\|\hat{E}\|_{2} \leq \xi$, $\|\hat{E}\|_F \leq \xi$.
Since $\xi = \frac{\epsilon}{2\epsilon+4}  < 1$, by lemma \ref{QI:TLS:Lemma:ERI}, it yields that
\begin{equation}\label{QI:TLS:Eq:RIF}
    \|R - I\|_F \leq \frac{\xi} {\sqrt{2(1 - \xi)}}.
\end{equation}
Since $0 < \xi < \frac{1}{2}$, we have $\frac{1} {\sqrt{2(1 - \xi)}} < 1$. The Eq.\eqref{QI:TLS:Eq:RIF} becomes
$\|R - I\|_F \leq \frac{\xi} {\sqrt{2(1 - \xi)}} < \xi$.

Setting $Z = Q\left[\begin{array}{c}
I \\
\mathbf{0}
\end{array}\right] \in \mathbb{R}^{(n+1) \times l}$, it follows that
\begin{equation}\label{QI:TLS:Eq:VZF}
    \|\hat{V} - Z\|_F  = \|R - I\|_F  < \xi.
\end{equation}
\end{proof}

\begin{Lemma}\label{QI:TLS:Theorem:CCZZ}
Given a matrix $A \in \mathbb{R}^{m \times n}$ and a vector $b \in \mathbb{R}^{m}$ satisfying the sample model and data structure, the parameters $(\epsilon, \delta, k)$ in the specified range of QiSVD algorithm. QiSVD algorithm outputs the approximate right singular matrix $\hat{V} \in \mathbb{R}^{(n+1) \times l}$, then with probability $1- \delta$, for $j \in [l]$, there exists a column orthonormal matrix $Z \in \mathbb{R}^{(n+1) \times l}$ such that
\begin{equation}\label{QI:TLS:Eq:U10}
    \|C - CZ_{j}Z_{j}^{T}\|^2_F  < \|C - C_{(j)}\|^2_F + 10 \epsilon \|C\|^2_F,
\end{equation}
where $C = [A,\,b]$, $C_{(j)} = \sum_{i= 1}^{j} \sigma_{i} u_{i} v_{i}^{T}$ and $Z_{j} = Z_{:,1:j}$.
\end{Lemma}

\begin{proof}
Based on lemma \ref{QI:TLS:Theorem:VZ}, denote $\hat{V}_{j} = \hat{V}_{:,1:j}$, for any $j \in [l]$, we have  $\|\hat{V}_{j} - Z_{j}\|_F < \xi$. Setting $\hat{V}_{j} = Z_{j} + \bar{E}$, with $\|\bar{E}\|_F < \xi$, using Eq.\eqref{QI:TLS:Eq:normV2}, it yields that

\begin{equation}\label{QI:TLS:Eq:CZZ}
\begin{aligned}
\|C - CZ_{j} Z_{j}^{T}\|_F
& = \|C - C(\hat{V}_j - \bar{E})(\hat{V}_j - \bar{E})^{T}\|_F \\
& \leq \|C - C\hat{V}_{j} \hat{V}^{T}_{j}\|_F + \|C\bar{E} \hat{V}^{T}_{j}\|_F + \|C \hat{V}_{j} \bar{E}^{T}\|_F + \|C\bar{E}\bar{E}^{T}\|_F \\
& < \|C - C\hat{V}_j \hat{V}_j^{T}\|_F + (2\xi \sqrt{1+\xi} + \xi^2) \|C\|_F.
\end{aligned}
\end{equation}
Squaring both sides of the Eq.\eqref{QI:TLS:Eq:CZZ}, we have
\begin{equation}\label{QI:TLS:Eq:CCZjZj}
\begin{aligned}
&~~~~\|C - CZ_{j}Z_{j}^{T}\|^2_F \\
& < \left[\|C - C\hat{V}_j \hat{V}_j^{T}\|_F + (2\xi \sqrt{1+\xi} + \xi^2) \|C\|_F \right]^2 \\
& = \|C - C\hat{V}_j \hat{V}_j^{T}\|^2_F + 2(2\xi \sqrt{1+\xi} + \xi^2)\|C - C\hat{V}_j \hat{V}_j^{T}\|_F \|C\|_F + (2\xi \sqrt{1+\xi} + \xi^2)^2 \|C\|^2_F.
\end{aligned}
\end{equation}
Using the fact that $\|C - C\hat{V}_j \hat{V}_j^{T}\|_F \leq \|C\|_F + \|C\|_F \|\hat{V}_j\|^2_{2}$
and Eq.\eqref{QI:TLS:Eq:normV2}, we have $\|C - C\hat{V}_j \hat{V}_j^{T}\|_F \leq (2+\xi)\|C\|_F$, applying lemma \ref{QI:TLS:Lemma:CVVFCKE}, we obtain
\begin{equation}\label{QI:TLS:Eq:AUUA}
\begin{aligned}
&~~~~\|C - CZ_{j}Z_{j}^{T}\|^2_F \\
& < \|C - C_{(j)}\|^2_F + \epsilon \|C\|_F^2 + 2 (2\xi \sqrt{1+\xi} + \xi^2)(2+\xi)\|C\|^2_F + (2\xi \sqrt{1+\xi} + \xi^2)^2 \|C\|^2_F   \\
& = \|C - C_{(j)} \|^2_F + \left[\epsilon + 2(2+\xi) (2\xi \sqrt{1+\xi} + \xi^2) + (2\xi \sqrt{1+\xi} + \xi^2)^2\right] \|C\|^2_F.
\end{aligned}
\end{equation}

Since $\xi = \frac{\epsilon}{2\epsilon +4}  < \frac{1}{2}$, it holds that
\begin{equation}\label{QI:TLS:Eq:tri}
\begin{aligned}
& ~~~~ \epsilon + 2(2+\xi) (2\xi \sqrt{1+\xi} + \xi^2) + (2\xi \sqrt{1+\xi} + \xi^2)^2 \\
& = \epsilon + \xi^4 + 6\xi^3 + 8\xi^2 + 4\xi^3 \sqrt{1+\xi} + 4\xi^2 \sqrt{1+\xi} + 8\xi \sqrt{1+\xi} \\
& < \epsilon + 15\xi + 16\xi \sqrt{1+\xi}
< \epsilon + 36\xi
= \frac{\epsilon^2 + 20\epsilon}{\epsilon+2}
< 10\epsilon.
\end{aligned}
\end{equation}
Then the results follows from Eqs.\eqref{QI:TLS:Eq:AUUA} and \eqref{QI:TLS:Eq:tri}.
\end{proof}

\begin{Theorem}\label{QI:TLS:Theorem:Vvev}
Suppose that a matrix $A \in \mathbb{R}^{m \times n}$ and a vector $b \in \mathbb{R}^{m}$ satisfying the sample model and data structure, the parameters $(\epsilon, \delta, k)$ in the specified range of QiSVD algorithm. Suppose that $C = [A,\,b]$ has singular values $\sigma_{1}, \sigma_{2}, \cdots, \sigma_{n + 1}$.
Assume that for $ l \leq q \leq k \leq n$, $\sigma^2_{i} - \sigma^2_{i+1} = \eta_{i}$, $i \in [q]$ and $\eta = {\rm min}\{\eta_{1}, \eta_{2}, \cdots, \eta_{q}\}$. Suppose that $V$ is the right singular matrix of $C$ defined in the introduction, and QiSVD algorithm outputs the approximate right singular matrix $\hat{V} \in \mathbb{R}^{(n+1) \times l}$, if $\eta \geq 20 \epsilon \|C\|^2_F$, then with probability $1- \delta$, it holds that
\begin{equation}\label{QI:TLS:Eq:VLV}
    \|V_{l} - \hat{V}\|_F \leq \sqrt{ \frac{40 k \epsilon}{\eta} } \|C\|_F + \xi,
\end{equation}
where $V_{l} = V_{:,1:l}$.
\end{Theorem}

\begin{proof}
The matrix $Z \in \mathbb{R}^{(n+1) \times l}$ builds a bridge to evaluate the upper bound of $\|V_l - \hat{V}\|_F $ in Eq.\eqref{QI:TLS:Eq:VLV}. In fact, we need to estimate two items $\|V_l - Z\|_F$ and $\|Z - \hat{V} \|_F$. The second is somewhat easier. For the first item, we need the trace of $Z^{T}V_l$, which depends on the estimate about $\|Z_{j}^{T} V_j\|_F^2$.

Since $C =  U \Sigma V^{T}$ and $\sigma^2_{i} - \sigma^2_{i+1} = \eta_{i}$, $i \in [q]$, based on lemma \ref{QI:TLS:Theorem:CCZZ}, for $j \in [l]$, we have
\begin{equation}\label{QI:TLS:Eq:AUUAE}
\begin{aligned}
   10 \epsilon \|C\|_F^2
   & > \|C - CZ_{j} Z_{j}^{T}\|^2_F - \|C - C_{(j)}\|^2_F\\
   & = (\|C\|^2_F - \|CZ_{j}\|^2_F) - (\|C\|^2_F -\|C_{(j)}\|^2_F)\\
   & = \|C_{(j)}\|^2_F - \|CZ_{j}\|^2_F \\
   &  = \|C_{(j)}\|^2_F - \|Z_{j}^{T} V \Sigma^{T} U^{T}\|^2_F \\
   & = \sum_{i=1}^{j}\sigma^2_{i} - \sum_{i=1}^{n+1}\sigma^2_{i} \|Z_{j}^{T} V_{:,i}\|^2_{2} \\
   & = \sum_{i=1}^{j}\sigma^2_{i}(1 - \|Z_{j}^{T} V_{:,i} \|^2_{2}) - \sum_{i=j+1}^{n+1}\sigma^2_{i}\| Z_{j}^{T} V_{:,i} \|^2_{2}\\
   & \geq \sigma^2_{j}\sum_{i=1}^{j}(1 - \| Z_{j}^{T} V_{:,i} \|^2_{2}) - \sigma^2_{j+1}\sum_{i=j+1}^{n+1}\| Z_{j}^{T} V_{:,i} \|^2_{2}\\
   & = \sigma^2_{j}(j - \| Z_{j}^{T} V_j\|^2_F) - \sigma^2_{j+1}\left(\| Z_{j}^{T} V\|^2_F - \| Z_{j}^{T} V_j\|^2_F \right)\\
   & = \sigma^2_{j}(j - \| Z_{j}^{T} V_j\|^2_F) - \sigma^2_{j+1}\left(j - \| Z_{j}^{T} V_j\|^2_F \right)\\
   & = \eta_j (j - \|Z_{j}^{T} V_j \|^2_F),
\end{aligned}
\end{equation}
which implies that
\begin{equation}\label{QI:TLS:Eq:AUUAEE}
\begin{aligned}
   \|Z_{j}^{T} V_j\|_F^2  \geq j - \frac{10 \epsilon \|C\|_F^2}{\eta_j}.
\end{aligned}
\end{equation}

For $a,b \in [l]$, set $t_{a,b} = [(Z^{T})_{a,:} V_{:,b}]^2$. Based on Eq.\eqref{QI:TLS:Eq:AUUAEE}, we have $\sum_{a,b=1}^{j} t_{a,b} \geq j - \frac{10 \epsilon \|C\|_F^2}{\eta_j}$. It yields that $\sum_{a,b=1}^{j-1} t_{a,b} \geq j-1 - \frac{10 \epsilon \|C\|_F^2}{\eta_{j-1}}$. Adding the two inequality, we gain
\begin{equation}\label{QI:TLS:Eq:UUUT11}
\sum_{a,b=1}^{j-1} t_{a,b} + \sum_{a,b=1}^{j} t_{a,b} \geq  2j-1 - \left(\frac{1}{\eta_j} + \frac{1}{\eta_{j-1}}\right)10 \epsilon \|C\|_F^2.
\end{equation}

Due to $\|Z_{j}^{T} V_{:,j}\|^2_{2} \leq \|Z_{j}^{T}\|^2_{2} \|V_{:,j}\|^2_{2} = 1$, we gain $\sum_{a=1}^{j} t_{a,b} \leq 1$.
$\forall \, b \in [l]$, it follows that
\begin{equation}\label{QI:TLS:Eq:UUUT22}
- \sum_{b=1}^{j-1}\sum_{a=1}^{j} t_{a,b} \geq 1 - j.
\end{equation}

Since $\|(Z_{:,i})^{T}V_{j}\|^2_{2} \leq \|(Z_{:,i})^{T}\|^2_{2} \|V_{j}\|^2_{2} = 1$, we have $\sum_{b=1}^{j} t_{a,b} \leq 1$.
$\forall \, a \in [l]$, it holds that
\begin{equation}\label{QI:TLS:Eq:UUUT223}
- \sum_{a=1}^{j-1}\sum_{b=1}^{j} t_{a,b} \geq 1 - j.
\end{equation}
Combining Eqs.\eqref{QI:TLS:Eq:UUUT11}, \eqref{QI:TLS:Eq:UUUT22} with \eqref{QI:TLS:Eq:UUUT223}, we have
\begin{equation}\label{QI:TLS:Eq:UUUT3}
1 - \left(\frac{1}{\eta_j} + \frac{1}{\eta_{j-1}}\right)10 \epsilon \|C\|_F^2
\leq \sum_{a,b=1}^{j-1} t_{a,b} + \sum_{a,b=1}^{j} t_{a,b} - \sum_{b=1}^{j-1}\sum_{a=1}^{j} t_{a,b} - \sum_{a=1}^{j-1}\sum_{b=1}^{j} t_{a,b}    = t_{j,j},
\end{equation}
which implies the following relation $[(Z^{T})_{j,:} V_{:,j}]^2 = t_{j,j} \geq 1 - \left(\frac{1}{\eta_j} + \frac{1}{\eta_{j-1}}\right)10 \epsilon \|C\|_F^2$. Since $\eta = {\rm min}\{\eta_{1}, \eta_{2}, \cdots, \eta_{q}\}$ and $\eta \geq 20 \epsilon \|C\|^2_F$, we have
$$(Z^{T})_{j,:} V_{:,j} = \sqrt{t_{j,j}} \geq \sqrt{1 - \left(\frac{1}{\eta_j} + \frac{1}{\eta_{j-1}}\right)10 \epsilon \|C\|_F^2} \geq \sqrt{ 1 -  \frac{20 \epsilon \|C\|_F^2}{\eta}} \geq  1 -  \frac{20 \epsilon \|C\|_F^2}{\eta}.$$
By induction method, we can compute the low bound on ${\rm Tr}(Z^{T}V_l)$, it holds that
\begin{equation}\label{QI:TLS:Eq:UUUT1}
\begin{aligned}
{\rm Tr}(Z^{T}V_l)
& = (Z^{T})_{1,:} V_{:,1} + (Z^{T})_{2,:} V_{:,2} + \cdots + (Z^{T})_{l,:} V_{:,l}  \\
& \geq \left[1 -  \frac{20 \epsilon \|C\|_F^2}{\eta}\right] + \left[1 -  \frac{20 \epsilon \|C\|_F^2}{\eta}\right] + \cdots + \left[1 -  \frac{20 \epsilon \|C\|_F^2}{\eta}\right] \\
& = l -  \frac{20l \epsilon \|C\|_F^2}{\eta}.
\end{aligned}
\end{equation}

By a direct computation, we obtain
\begin{equation}\label{QI:TLS:Eq:VLZF}
\|V_l - Z\|_F^2
= {\rm Tr} [(V_l - Z)^{T}(V_l - Z)] 
= 2l - 2 {\rm Tr}(Z^{T}V_l)
\leq 2l - 2\left[l -  \frac{20l \epsilon \|C\|_F^2}{\eta}\right] 
= \frac{40l \epsilon \|C\|_F^2}{\eta}.
\end{equation}

Since $l \leq k$, applying Eqs.\eqref{QI:TLS:Eq:VZF} and \eqref{QI:TLS:Eq:VLZF}, we obtain
\begin{equation}\label{QI:TLS:Eq:UUUT0}
\begin{aligned}
\| V_l - \hat{V} \|_F
\leq \| V_l - Z \|_F + \|Z - \hat{V} \|_F
\leq \sqrt{ \frac{40 k \epsilon}{\eta} } \|C\|_F + \xi.
\end{aligned}
\end{equation}

\end{proof}

\begin{Theorem}\label{QI:TLS:Theorem:UI}
Suppose that a matrix $A \in \mathbb{R}^{m \times n}$ and a vector $b \in \mathbb{R}^{m}$ satisfying the sample model and data structure, the parameters $(\epsilon, \delta, k)$ in the specified range of QiSVD algorithm . Assume $m \geq n + 1$ and $d \leq l \leq q \leq k \leq n$. Suppose that $C = [A,\,b]$ has singular values $\sigma_{1}, \sigma_{2}, \cdots, \sigma_{n + 1}$ and $A$ has singular values $\sigma^{A}_{1}, \sigma^{A}_{2}, \cdots, \sigma^{A}_{n}$. Assume that $\sigma^2_i - \sigma^2_{i+1} = \eta_{i}$, $i \in [q ]$, $\eta = {\rm min}\{\eta_1, \eta_2, \cdots, \eta_q\}$ and $\sigma^{A}_q > \sigma_{q+1}$.
QiTTLS and TTLS algorithms output the approximate TTLS solution $x_{\rm QiTTLS} = (\hat{\mathbf{V}}_{11}^T)^{\dag}\hat{\mathbf{v}}_{21}^{T}$ and $x_{\rm TTLS} = (\mathbf{V}_{11}^T)^{\dag} \mathbf{v}_{21}^{T}$, respectively.
Assume that $\mathbf{V}_{11} \in \mathbb{R}^{n \times d}$ has singular values $\tau_1, \tau_2, \cdots, \tau_d$. If $\eta \geq 20 \epsilon \|C\|^2_F$, $\|x_{\rm TTLS}\|_{2} \neq 0$, $\tau_{d} > \epsilon_v$ and $\|b\|_{2} > \sigma_{d+1}$, then with probability $1- \delta$, it holds that
\begin{equation}\label{QI:TTLS:Eq:xqx}
\frac{\| x_{\rm TTLS} - x_{\rm QiTTLS}\|_{2}}{\|x_{\rm TTLS}\|_{2}}
\leq \frac{(\sqrt{2} \epsilon_{v} + \tau_{d} + 1 )}{\tau_{d} - \epsilon_v}  \frac{2 \sigma_1}{\|b\|_{2} - \sigma_{d+1}},
\end{equation}
where $\epsilon_v = \sqrt{ \frac{40 k \epsilon}{\eta} } \|C\|_F + \xi$.
\end{Theorem}

\begin{proof}
Set $\mathbf{V}_{11} = \hat{\mathbf{V}}_{11} + \check{E}$ and $\mathbf{v}_{21}^{T} = \hat{\mathbf{v}}_{21}^{T} + e$, respectively.
Let $A, B \in \mathbb{R}^{m \times n}$, according to \cite{PAW73}, it yields that
$$B^{\dag} - A^{\dag} = -B^{\dag}(B-A)A^{\dag} + B^{\dag}(I - AA^{\dag}) - (I - B^{\dag}B)A^{\dag}.$$
By a direct computation, we obtain
\begin{equation}\label{QI:TTLS:Eq:vv2vv0}
    \begin{aligned}
    & \ \ x_{\rm TTLS} - x_{\rm QiTTLS} \\
    & = (\hat{\mathbf{V}}_{11}^T)^{\dag}\hat{\mathbf{v}}_{21}^{T} - (\mathbf{V}_{11}^T)^{\dag} \mathbf{v}_{21}^{T}\\
    & = (\hat{\mathbf{V}}_{11}^T)^{\dag}(\mathbf{v}_{21}^{T} - e) - (\mathbf{V}_{11}^T)^{\dag} \mathbf{v}_{21}^{T}\\
    & = [(\hat{\mathbf{V}}_{11}^T)^{\dag} -(\mathbf{V}_{11}^T)^{\dag}]\mathbf{v}_{21}^{T} - (\hat{\mathbf{V}}_{11}^T)^{\dag} e \\
    & = -(\hat{\mathbf{V}}_{11}^T)^{\dag} \check{E}^T (\mathbf{V}_{11}^T)^{\dag}\mathbf{v}_{21}^{T} + (\hat{\mathbf{V}}_{11}^T)^{\dag} [I - \mathbf{V}_{11}^T {(\mathbf{V}_{11}^T)}^{\dag}]\mathbf{v}_{21}^{T} - [ I - {(\hat{\mathbf{V}}_{11}^T)}^{\dag} \hat{\mathbf{V}}_{11}^T] (\mathbf{V}_{11}^T)^{\dag}\mathbf{\mathbf{v}}_{21}^{T} - (\hat{\mathbf{V}}_{11}^T)^{\dag} e\\
    & = -(\hat{\mathbf{V}}_{11}^T)^{\dag} \check{E}^T x_{\rm TTLS} + (\hat{\mathbf{V}}_{11}^T)^{\dag} [I - \mathbf{V}_{11}^T {(\mathbf{V}_{11}^T)}^{\dag}]\mathbf{v}_{21}^{T} - [I -  {(\hat{\mathbf{V}}_{11}^T)}^{\dag}\hat{\mathbf{V}}_{11}^T] x_{\rm TTLS} - (\hat{\mathbf{V}}_{11}^T)^{\dag} e.
    \end{aligned}
\end{equation}

Based on $\sigma^{A}_q > \sigma_{q+1}$, it yields that $\mathbf{V}_{11}$ is full column rank, we have
$$I - \mathbf{V}_{11}^T  {(\mathbf{V}_{11}^T)}^{\dag}  = I -   \mathbf{V}_{11}^T  \mathbf{V}_{11} {(\mathbf{V}_{11}^T \mathbf{V}_{11}})^{-1} = 0,$$
and the Eq.\eqref{QI:TTLS:Eq:vv2vv0} becomes
\begin{equation}\label{QI:TTLS:Eq:vv2vv1}
x_{\rm TTLS} - x_{\rm QiTTLS} = -(\hat{\mathbf{V}}_{11}^T)^{\dag} \check{E}^T x_{\rm TTLS} - x_{\rm TTLS} +  {(\hat{\mathbf{V}}_{11}^T)}^{\dag}\hat{\mathbf{V}}_{11}^T x_{\rm TTLS} - (\hat{\mathbf{V}}_{11}^T)^{\dag} e.
\end{equation}

Based on the theorem \ref{QI:TLS:Theorem:Vvev}, $\hat{\mathbf{V}}_{11} = \hat{V}_{1:n,1:d}$ and $\hat{\mathbf{v}}_{21}= \hat{V}_{n+1,1:d}$, it yields that
\begin{equation}\label{QI:TTLS:Eq:Eekv}
    \epsilon^2_{v} \geq \|V_l - \hat{V}\|^2_F  \geq \|\mathbf{V}_{11} - \hat{\mathbf{V}}_{11}\|^2_F + \|\mathbf{v}_{21} - \hat{\mathbf{v}}_{21}\|^2_{2} = \|\check{E}\|^2_F + \|e\|^2_{2}.
\end{equation}

Assume that $\hat{\mathbf{V}}_{11} \in \mathbb{R}^{n \times d}$  has singular values $\hat{\sigma}_{1}, \hat{\sigma}_{2}, \cdots, \hat{\sigma}_{d}$. Using the lemma \ref{QI:SLSMC:Lemma:SmSn} and Eq.\eqref{QI:TTLS:Eq:Eekv}, for any $i \in [d]$, we have
\begin{equation}\label{QI:TTLS:Eq:sigmav2v}
|\tau_{i}(\mathbf{V}^{T}_{11}) - \hat{\sigma}_{i}(\hat{\mathbf{V}}^{T}_{11})| \leq \|\mathbf{V}^{T}_{11} - \hat{\mathbf{V}}^{T}_{11}\|_{2} = \|\mathbf{V}_{11} - \hat{\mathbf{V}}_{11}\|_{2} \leq \|\mathbf{V}_{11} - \hat{\mathbf{V}}_{11}\|_F \leq  \epsilon_v,
\end{equation}
which implies the following relationship
$$\hat{\sigma}_{1}(\hat{\mathbf{V}}^{T}_{11}) \leq \tau_{1}(\mathbf{V}^{T}_{11}) +  \epsilon_v \leq 1 +  \epsilon_v$$
and
$$\hat{\sigma}_{d}(\hat{\mathbf{V}}^{T}_{11}) \geq \tau_{d}(\mathbf{V}^{T}_{11})- \epsilon_v = \tau_{d} -  \epsilon_v.$$
Moreover, it yields that $\|(\hat{\mathbf{V}}_{11}^T)^{\dag}\|_{2} \leq \frac{1}{\tau_{d} -  \epsilon_v}$ and $\|\hat{\mathbf{V}}_{11}^T\|_{2} \leq 1 +  \epsilon_v$.

By directly computation, we check that
\begin{equation}\label{QI:TTLS:Eq:x2qx}
    \begin{aligned}
& \ \ \|x_{\rm TTLS} - x_{\rm QiTTLS}\|_{2}  \\
& = \|(\hat{\mathbf{V}}_{11}^T)^{\dag}\hat{\mathbf{v}}_{21}^{T} - (\mathbf{V}_{11}^T)^{\dag} \mathbf{v}_{21}^{T}\|_{2} \\
& = \|-(\hat{\mathbf{V}}_{11}^T)^{\dag} \check{E}^T x_{\rm TTLS} - x_{\rm TTLS} +  {(\hat{\mathbf{V}}_{11}^T)}^{\dag}\hat{\mathbf{V}}_{11}^T x_{\rm TTLS} - (\hat{\mathbf{V}}_{11}^T)^{\dag} e\|_{2} \\
& \leq \|(\hat{\mathbf{V}}_{11}^T)^{\dag}\|_{2} \|\check{E}^T\|_F  \|x_{\rm TTLS}\|_{2} + \|x_{\rm TTLS}\|_{2} + \|{(\hat{\mathbf{V}}_{11}^T)}^{\dag}\|_{2} \|\hat{\mathbf{V}}_{11}^T\|_{2}  \|x_{\rm TTLS}\|_2 + \|(\hat{\mathbf{V}}_{11}^T)^{\dag}\|_{2}  \|e\|_{2} \\
& \leq  \frac{\|\check{E}\|_F}{\tau_{d} - \epsilon_v}  \|x_{\rm TTLS}\|_{2} + \|x_{\rm TTLS}\|_{2}  +  \frac{1 +  \epsilon_v}{\tau_{d} -  \epsilon_v} \|x_{\rm TTLS}\|_{2} + \frac{\|e\|_{2}}{\tau_{d} - \epsilon_v} \\
& =  \frac{\|\check{E}\|_F + 1 + \tau_{d}}{\tau_{d} - \epsilon_v}  \|x_{\rm TTLS}\|_{2} + \frac{\|e\|_{2}}{\tau_{d} -  \epsilon_v} \\
& \leq \sqrt{ \left[\left(\frac{\|\check{E}\|_F + \tau_{d} + 1}{\tau_{d} - \epsilon_v}\right)^2 + \left(\frac{\|e\|_{2}}{\tau_{d} -  \epsilon_v}\right)^2 \right] (\|x_{\rm TTLS}\|^2_{2} + 1)} \\
& \leq \frac{\|\check{E}\|_F + \tau_{d} + 1 + \|e\|_{2}}{\tau_{d} -  \epsilon_v} \sqrt{\|x_{\rm TTLS}\|^2_{2} + 1}\\
& \leq \frac{\sqrt{2} \epsilon_{v} + \tau_{d} + 1}{\tau_{d} - \epsilon_v} \sqrt{\|x_{\rm TTLS}\|^2_{2} + 1}. \ \ \ \ ({\rm using}\, {\rm the}\, {\rm result}\,{\rm of}\,{\rm Eq}.\eqref{QI:TTLS:Eq:Eekv}) \\
    \end{aligned}
\end{equation}

Using the result of theorem 4.1 in \cite{MSW92}, the Eq.\eqref{QI:TTLS:Eq:x2qx} can be rewritten as
\begin{equation}\label{QI:TTLS:Eq:x2qx1}
    \begin{aligned}
\frac{\|x_{\rm TTLS} - x_{\rm QiTTLS}\|_{2}}{\|x_{\rm TTLS}\|_{2}}
& \leq \frac{(\sqrt{2} \epsilon_{v} + \tau_{d} + 1 )}{\tau_{d} -  \epsilon_v}  \frac{\sqrt{\|x_{\rm TTLS}\|^2_{2} + 1}}{\|x_{\rm TTLS}\|_{2}}\\
& = \frac{(\sqrt{2} \epsilon_{v} + \tau_{d} + 1 )}{\tau_{d} -  \epsilon_v}  \frac{1}{\|\mathbf{v}_{21}\|_{2}} \\
& \leq \frac{(\sqrt{2} \epsilon_{v} + \tau_{d} + 1 )}{\tau_{d} - \epsilon_v}  \frac{2 \sigma_1}{\|b\|_{2} - \sigma_{d+1}}.
    \end{aligned}
\end{equation}
\end{proof}

Next, we will briefly describe the computational complexity of each step in our QiTTLS algorithm. The cost of each step of QiTTLS algorithm is listed as follows.

\begin{enumerate}[(1)]
\item QiSVD algorithm
\begin{itemize}
  \item In step \ref{QI:TLS:QiTTLS:al:s0}, the time of solving parameters $\xi$, $\alpha$, $\theta$ and $p$ can be neglected.
  \item In step \ref{QI:TLS:QiTTLS:al:s1}-\ref{QI:TLS:QiTTLS:al:s2}, based on the data structure, we samples $p$ row indies, each sampling takes no more than $O ({\rm log}\, m)$ flops.
  \item In step \ref{QI:TLS:QiTTLS:al:s3}-\ref{QI:TLS:QiTTLS:al:s4}, similarly, we need to sample $p$ column indies, each sampling takes no more than $O ({\rm log}\, n)$ flops.
  \item In step \ref{QI:TLS:QiTTLS:al:s5}, there is the SVD of $W \in \mathbb{R}^{p \times p}$, which costs about $O(p^3)$ flops.
  \item In step \ref{QI:TLS:QiTTLS:al:s6}, it needs about $O(p)$ flops.
  \item In step \ref{QI:TLS:QiTTLS:al:s7}, Calculate the approximate right singular values $\hat{V} = S^{T} \bar{U} \bar{\Sigma}^{-1} \in \mathbb{R}^{(n+1) \times l}$, which costs about $O(npl)$ flops.
\end{itemize}
\item The classical SVD based on R-bidiagonalization (\emph{R}-SVD) \cite{GL13}

We use the block form of the approximate singular vector $\hat{V}$ to form the solution, i.e.,$x_{\rm QiTTLS} = (\hat{\mathbf{V}}_{11}^T)^{\dag} \hat{\mathbf{v}}_{21}^{T}$. The Moore-Penrose inverse via \emph{R}-SVD \cite{GL13} needs about $2nd^2 + 11d^3$ flops.
\end{enumerate}
Above all, since $p  = \left\lceil\frac{1}{\theta^2 \delta}\right\rceil$, $\theta = \alpha \xi$, $\alpha = \frac{\xi}{16 k^4}$ and $\xi = \frac{\epsilon}{2\epsilon + 4}$, we have $p = \left\lceil\frac{16^2 k^8 (2\epsilon+4)^4}{\epsilon^4 \delta}\right\rceil$. The total computational complexity of QiTTLS algorithm is the following
\begin{equation}\label{QI:TTLS:Eq:time3}
    \begin{aligned}
    & \ \ O(p {\rm log}\, m + p {\rm log}\, n + p^3 + npl + nd^2 + d^3)\\
    & = O\left(\frac{k^{8}}{\epsilon^{4} \delta} {\rm log}\, m + \frac{k^{16}}{\epsilon^{8} \delta^2} {\rm log}\, n + \frac{k^{24}}{\epsilon^{12} \delta^3} + \frac{k^{8}}{\epsilon^{4} \delta} nl + nd^2 + d^3 \right).
    \end{aligned}
\end{equation}

\section{Numerical experiments}\label{sec:QI:TLS:experiments}
In this section, we give several numerical examples to illustrate that the performance of the QiTTLS algorithm is as accurate as the classical methods, and make a comparison with the RTTLS algorithm in \cite{XXW19}. The numerical tests are performed on a laptop with Intel Core i5 by MATLAB R2016(a) with a machine precision of $10^{-16}$.

\begin{Example}\label{QI:TTLS:Exp1}
The ill-conditioned cases are taken from Hansen's regularization tools \cite{PCH07}. All the problems are derived from discretizations of Fredholm integral equations of the first kind with a square integrable kernel \cite{CWG84}
\begin{equation}\label{QI:TTLS:Numerical:Eq:fab}
    \int_{a}^{b} K(s,t)f(t)dt = g(s), \ \ c \leq s \leq d,
\end{equation}
where the right-hand side $g$ and the kernel $K$ are given, and $f$ is the unknown solution. There are two kinds of discretization methods, namely orthonormal basis function method and Galerkin method. Here we select the examples of foxgood, gravity, heat, phillips, baart, deriv2 as shown in Table \ref{QI:TLS:tab:8exm}. The decay trend of the corresponding singular values is shown in Figure \ref{fig:Dsigma}.
\end{Example}

\begin{table}[!htb]
    \caption{The 6 testing matrices in Hansen's regularization tools.}\label{QI:TLS:tab:8exm}
    \centering
	\begin{tabular}{l  l}
		\hline
		Matrix 	& Description(\cite{PCH07}) \\
		\hline
		Foxgood 	& Severely ill-posed test problem  \\
		Gravity 	& One-dimensional gravity surveying problem  \\
		Heat 	    & Inverse heat equation  \\
		Philips     & Phillips' famous test problem \\
		Baart 	    & Discretizations of the first kind Fredholm integral equations   \\
		Deriv2      & Computational of the second derivative \\
		\hline
	\end{tabular}
\end{table}

\begin{figure}[h]
	\centering
	\includegraphics[width=0.55\textwidth]{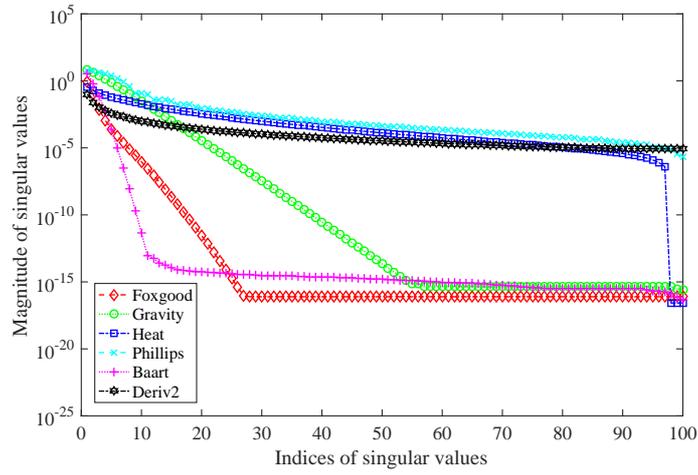}
	\caption{The decaying trends of singular values of the 6 testing problems. }
	\label{fig:Dsigma}
\end{figure}

For example, the above case Foxgood is generated by the MATLAB command $[\tilde{A}, \tilde{b}] = {\rm Foxgood}(m)$, where $m$ is the matrix size. Followed by the work in \cite{XZ13,XZ15,XXW19}, the observation data $A$ and $b$ are generated from the exact data $\tilde{A}$ and $\tilde{b}$ by adding noise $\eta$, respectively.
\begin{equation}\label{QI:TTLS:Numerical:Eq:AEb}
    A = \tilde{A} + \eta \|\tilde{A}\|_F \frac{G}{\|G\|_F}, \ \ b = \tilde{b} + \eta \|\tilde{b}\|_2 \frac{\zeta}{\|\zeta\|_2},
\end{equation}
where $G$ is a random matrix, i.e., $G = 2*{\rm rand}(m)-1$, and $\zeta$ is a random vector, i.e., $\zeta = 2*{\rm rand}(m,1)-1$ in MATLAB notations. Then, we will compute the TTLS solution of $Ax \approx b$ and compare it with the corresponding true solution $x_{\rm true}$ in the sense of time and accuracy.

\begin{figure}[H]
\begin{minipage}[t]{0.5\linewidth}
\centering
\includegraphics[width=2.8in]{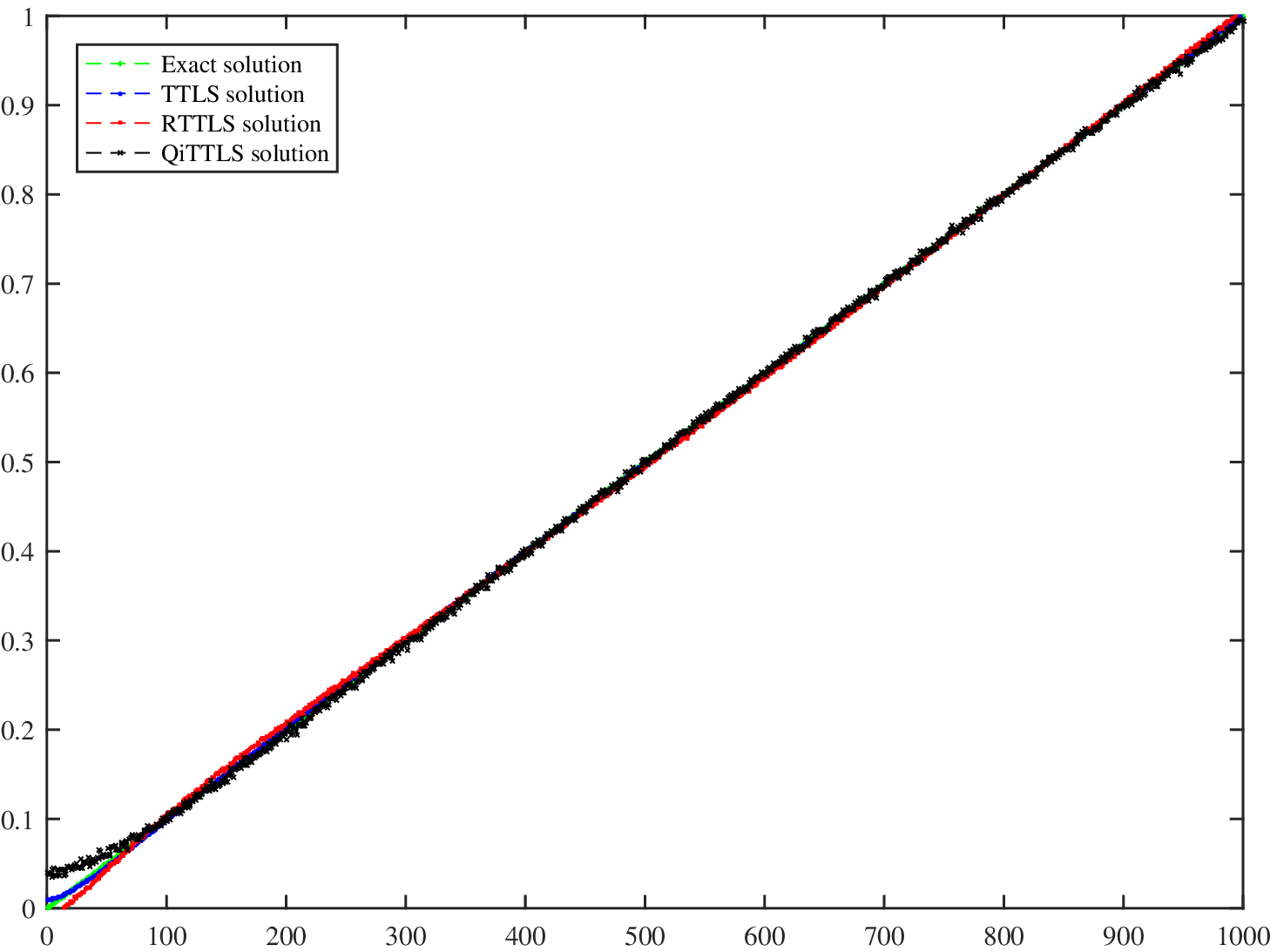}
\caption{Foxgood}
\label{fig:TTLS:fox}
\end{minipage}%
\begin{minipage}[t]{0.5\linewidth}
\centering
\includegraphics[width=2.8in]{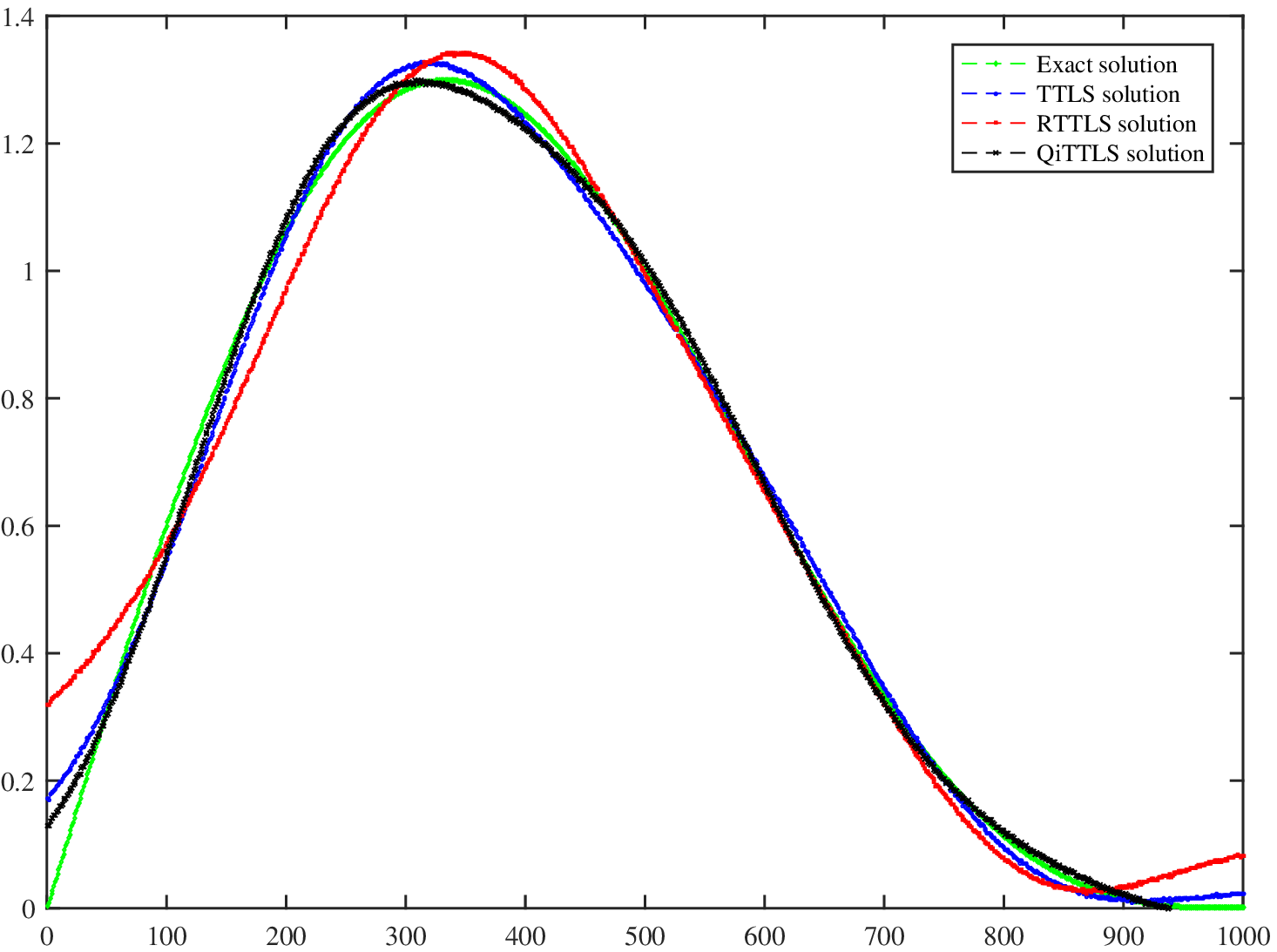}
\caption{Gravity}
\label{fig:TTLS:gra}
\end{minipage}
\end{figure}

\begin{figure}[H]
\begin{minipage}[t]{0.5\linewidth}
\centering
\includegraphics[width=2.8in]{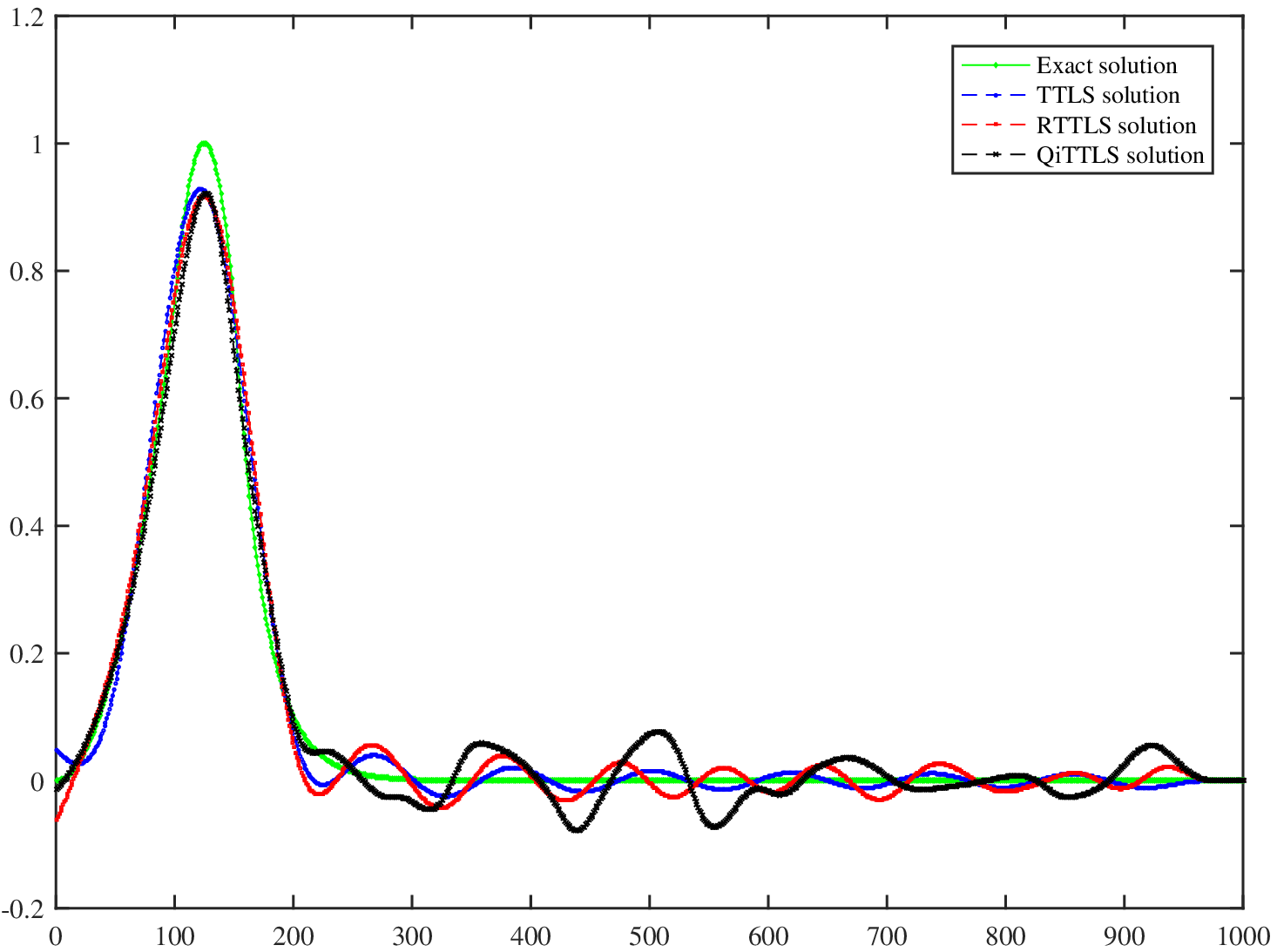}
\caption{Heat}
\label{fig:TTLS:heat}
\end{minipage}%
\begin{minipage}[t]{0.5\linewidth}
\centering
\includegraphics[width=2.8in]{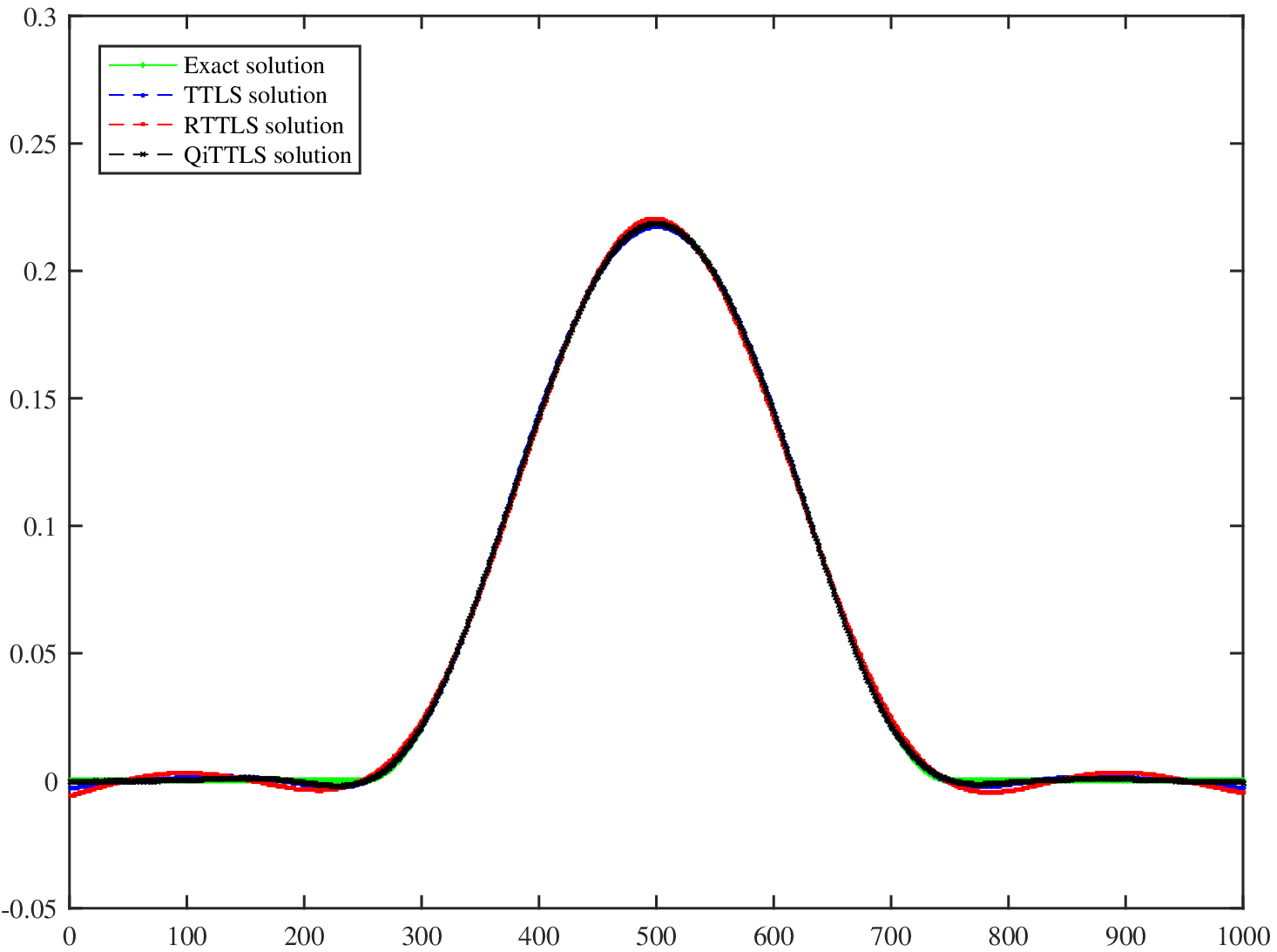}
\caption{Philips}
\label{fig:TTLS:phi}
\end{minipage}
\end{figure}

\begin{figure}[H]
\begin{minipage}[t]{0.5\linewidth}
\centering
\includegraphics[width=2.8in]{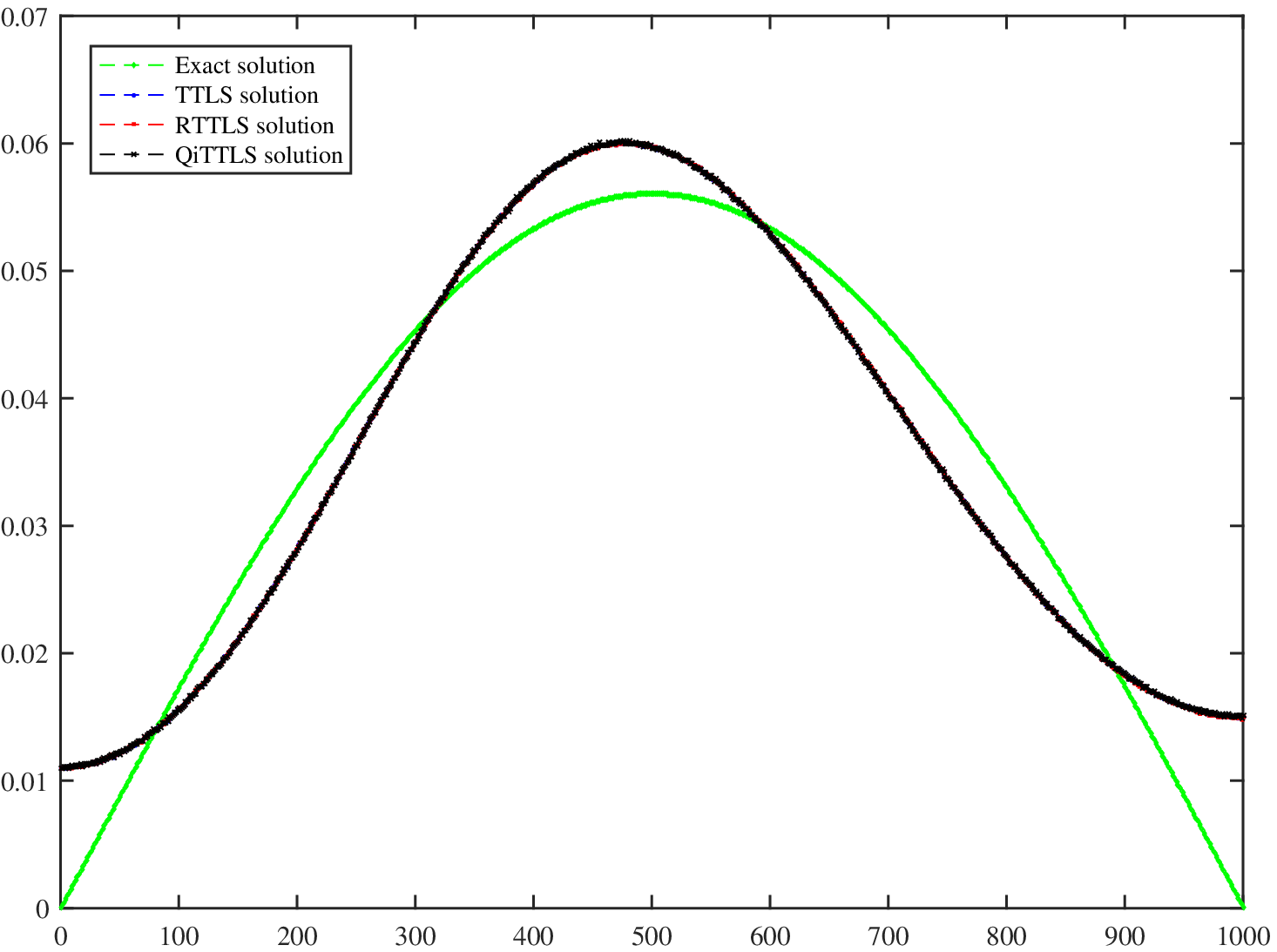}
\caption{Barrt}
\label{fig:TTLS:baat}
\end{minipage}%
\begin{minipage}[t]{0.5\linewidth}
\centering
\includegraphics[width=2.8in]{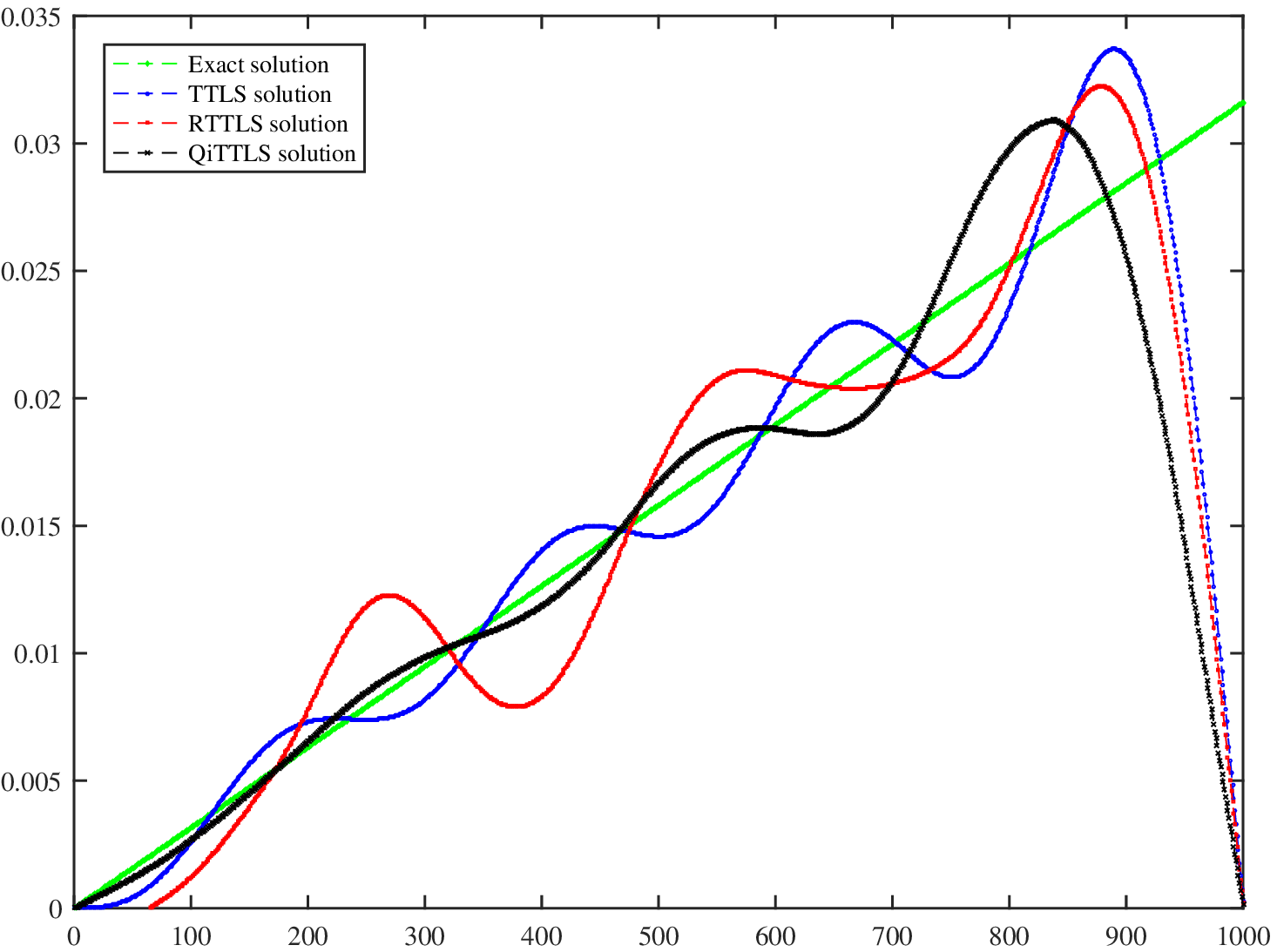}
\caption{Deriv2}
\label{fig:TTLS:deriv2}
\end{minipage}
\end{figure}

\begin{table}[!htb]
    \caption{Comparison with TTLS, RTTLS and QiTTLS algorithms of $A$ for $m = 1000$.}\label{QI:TTLS:tab:1comp}
    \centering
	\begin{tabular}{l c c c c c c c c c}
		\hline
		Matrix 	& $d$ & ${\rm Err}_{T}$ & ${\rm Time}_{T}$ & ${\rm Err}_{R}$ & ${\rm Time}_{R}$ & ${\rm Err}_{Qi}$ & ${\rm Time}_{Qi}$ \\
		\hline
		Foxgood 	& 4    & 1.830E-02  &  0.0345 & 3.330E-02  &  0.0270 & 1.985E-02  & 0.0224\\
		Gravity 	& 6    & 6.760E-02  &  0.0520 & 8.820E-02  &  0.0409 & 8.703E-02  & 0.0315\\
		Heat 	    & 15   & 9.370E-02  &  0.0920 & 1.226E-01  &  0.0583 & 9.351E-02  & 0.0487\\
		Philips     & 10   & 1.370E-02  &  0.0503 & 9.001E-03  &  0.0558 & 8.826E-03  & 0.0421\\
		Baart 	    & 4    & 2.654E-01  &  0.0493 & 2.642E-01  &  0.0271 & 2.614E-01  & 0.0258\\
		Deriv2      & 6    & 9.920E-01  &  0.0702 & 9.939E-01  &  0.0285 & 9.921E-01  & 0.0385\\
		\hline
	\end{tabular}
\end{table}

\begin{table}[!htb]
    \caption{Comparison with TTLS, RTTLS and QiTTLS algorithms of $A$ for $m = 4000$.}\label{QI:TTLS:tab:4comp}
    \centering
	\begin{tabular}{l c c c c c c c c c}
		\hline
		Matrix 	& $d$ & ${\rm Err}_{T}$ & ${\rm Time}_{T}$ & ${\rm Err}_{R}$ & ${\rm Time}_{R}$ & ${\rm Err}_{Qi}$ & ${\rm Time}_{Qi}$ \\
		\hline
		Foxgood 	& 5    & 1.390E-02  & 0.4133 & 8.410E-02  & 0.3113  & 7.792E-02  & 0.2554 \\
		Gravity 	& 6    & 1.361E-01  & 0.5127 & 1.717E-01  & 0.3304  & 1.610E-01  & 0.2917\\
		Heat 	    & 18   & 8.160E-02  & 0.8111 & 8.610E-02  & 0.8527  & 8.527E-02  & 0.5070\\
		Philips     & 11   & 1.370E-02  & 0.8291 & 1.850E-02  & 0.4614  & 1.763E-02  & 0.3789\\
		Baart 	    & 4    & 2.654E-01  & 0.3455 & 2.668E-01  & 0.3381  & 2.659E-01  & 0.3074\\
		Deriv2      & 6    & 9.980E-01  & 0.4495 & 9.987E-01  & 0.3976  & 9.981E-01  & 0.2936\\
		\hline
	\end{tabular}
\end{table}

For the ill-conditioned cases from the Hansen's Regularization Tools \cite{PCH07}, the numerical results of TTLS, RTTLS and QiTTLS algorithms are shown in Figures \ref{fig:TTLS:fox}-\ref{fig:TTLS:deriv2}. Since the true solutions of these problems are known in advance, we compute the relative errors by
\begin{equation*}
\begin{aligned}
{\rm Err}_{T} & = \| x_{\rm{TTLS}} -  x_{\rm{true}}\|_{\infty} / \| x_{\rm{true}}\|_{\infty},\\
{\rm Err}_{R} & = \| x_{\rm{RTTLS}} -  x_{\rm{true}}\|_{\infty} / \| x_{\rm{true}}\|_{\infty},\\
{\rm Err}_{Qi} & = \| x_{\rm{QiTTLS}} -  x_{\rm{true}}\|_{\infty} / \| x_{\rm{true}}\|_{\infty}.
\end{aligned}
\end{equation*}
Except for comparing ${\rm Err}_{T}$, ${\rm Err}_{R}$ and ${\rm Err}_{Qi}$, we also list the performances of TTLS, RTTLS and QiTTLS in terms of computing time in seconds (denoted by ${\rm Time}_{T}$, ${\rm Time}_{R}$ and ${\rm Time}_{Qi}$), where ${\rm Time}_{T}$, ${\rm Time}_{R}$ and ${\rm Time}_{Qi}$ are given by the MATLAB functions tic and toc. Besides, we know that the sampling parameters need to be chosen sufficiently large so that it can improve the performance. For simplicity, here we set the sampling parameter size $l = 20$ for all cases in the RTTLS algorthm. From the Tables
\ref{QI:TTLS:tab:1comp}-\ref{QI:TTLS:tab:4comp}, compared with TTLS and RTTLS algorithms, we obtain that QiTTLS algorithm can achieve a better accuracy and cost less time.

\begin{Example}\label{QI:TTLS:Exp2}
The TLS method is a promising approach in dealing with signal processing. In \cite{RY87}, Rahman and Yu propose a frequency estimation method using the TLS algorithm to solve linear prediction equations. Here we consider a set of linear prediction equations, which comes from the work of Majda et al. \cite{MSW89}. Suppose $a_{j} = [y_{j-1}, \cdots, y_{j + m - 2}]^{T}$, where $y_{l} = \sum_{j=1}^{p} \gamma _{j} z_{j}^{l}$, $z_{j} = e^{\lambda_{j}t}$, $j = 1, 2, \cdots, p$, the parameters $\gamma_{j}$ and $\lambda_{j}$ also need to be determined. Moreover, suppose $\gamma _{j}$ and $z_{j}$ are nonzeros, $z_j$ are different for $j = 1, 2, \cdots, p$. Denote $A_{n} = [a_{1}, a_{2}, \cdots, a_{n}]$, $b_{n} = - a_{n+1}$ and consider the following linear system
\begin{equation}\label{QI:TTLS:Numerical:Eq:AXB}
    A_{n} x = b_{n}.
\end{equation}
Assume that $m \geq n$, $m \geq p$. we know that ${\rm rank}(A_{n}) = {\rm min}\{n, p\}$. If $n \geq p$, then the linear system \eqref{QI:TTLS:Numerical:Eq:AXB} is compatible.
\end{Example}

\begin{table}[H]
    \caption{Six pairs of poles and residues \cite{MSW89}.}\label{QI:TTLS:tab:Ex2:cs}
    \centering
	\begin{tabular}{c c}
		\hline
		$\lambda_j$ & $\gamma_{j}$ \\
		\hline
		$-0.082 \pm 0.926i$ 	& 1  \\
		$-0.147 \pm 2.874i$ 	& 1  \\
		$-0.188 \pm 4.835i$ 	& 1  \\
        $-0.220 \pm 6.800i$	    & 1  \\
        $-0.247 \pm 8.767i$ 	& 1  \\
        $-0.270 \pm 10.733i$ 	& 1  \\
		\hline
	\end{tabular}
\end{table}

\begin{figure}[H]
	\centering
	\includegraphics[width=0.5\textwidth]{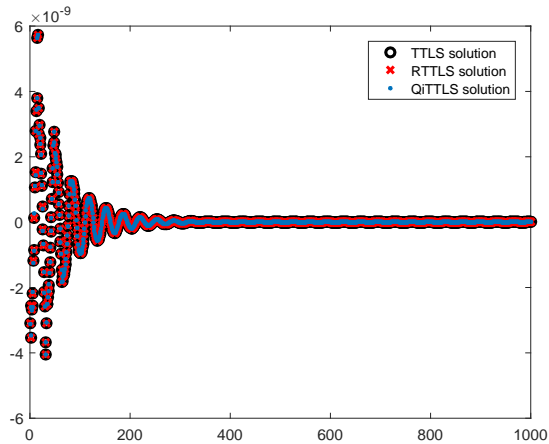}
	\caption{Computed RTTLS and QiTTLS solutions of Prony with $m = 1000$. }
	\label{fig:Prony}
\end{figure}

\begin{table}[H]
    \caption{Comparison with RTTLS and QiTTLS algorithms of $A$ for different $m$.}\label{QI:TTLS:tab:Ex2:2comp}
    \centering
	\begin{tabular}{c c c c c c c c c}
		\hline
		$m$ 	& $d$ & ${\rm Err}_{R}$ & ${\rm Time}_{R}$ & ${\rm Err}_{Qi}$ & ${\rm Time}_{Qi}$ \\
		\hline
		1000 	& 12 & 4.3408E-08 & 0.0364 & 6.9114E-10  & 0.0204 \\
		2000 	& 12 & 5.7792E-08 & 0.0422 & 1.0699E-08  & 0.0368 \\
		3000 	& 12 & 4.2830E-08 & 0.0468 & 7.9295E-09  & 0.0391 \\
        4000 	& 12 & 4.1737E-08 & 0.0697 & 7.7270E-09  & 0.0516 \\
		\hline
	\end{tabular}
\end{table}
For the Prony modeling example, in order to make a fair comparison with the RTTLS algorithm. Here the parameter settings are consistent with \cite{XXW19}, set the parameters $t = 0.2$, $d = 12$, $n = 1000$, with $\lambda_{j}$ and $\gamma_{j}$ as shown  in Table \ref{QI:TTLS:tab:Ex2:cs}. Since we do not know the exact solution of this problem, the relative errors are defined by
\begin{equation*}
    \begin{aligned}
    {\rm Err}_{R} & = \| x_{\rm{RTTLS}} - x_{\rm {TTLS}}\|_{\infty} / \|x_{\rm{TTLS}}\|_{\infty},\\
    {\rm Err}_{Qi} & = \|x_{\rm{QiTTLS}} - x_{\rm {TTLS}}\|_{\infty} / \|x_{\rm {TTLS}}\|_{\infty}.
    \end{aligned}
\end{equation*}
Except for comparing ${\rm Err}_{R}$ and ${\rm Err}_{Qi}$, we also list the performances of RTTLS and QiTTLS in terms of computing time in seconds (denoted by ${\rm Time}_{R}$ and ${\rm Time}_{Qi}$), where ${\rm Time}_{R}$ and ${\rm Time}_{Qi}$ are given by the MATLAB functions tic and toc. The results are given in Table \ref{QI:TTLS:tab:Ex2:2comp} and Figure \ref{fig:Prony}, we conclude that QiTTLS algorithm needs less time while maintaining the better accuracy.

\section{Conclusions}\label{sec:QI:TLS:conclusions}
In this paper, based on the sample model and data structure technique, we present a quantum-inspired total least squares algorithm for the large-scale ill-posed problem. It is a generalization of the method in \cite{FKV04,XXW19}. The proposed quantum-inspired truncated total least squares (QiTTLS) algorithm is essentially a randomized SVD algorithm. Next, we theoretically analyze the approximation accuracy and the computational complexity of our algorithm. Combined with numerical experiments, we show that the QiTTLS algorithm is competitive with the randomized TTLS algorithm. In the future, there are several improvements on our QiTTLS algorithm, such as choosing the other regularization parameter techniques, tightening the upper bounds in the theoretical analysis, and reducing the computational complexity.

\bibliographystyle{model1-num-names}

\end{document}